\documentclass[11pt,twoside,reqno]{amsart}

\usepackage{amssymb, latexsym}
\usepackage[italian,english]{babel}
\usepackage{amsmath,amsfonts,amsthm}
\usepackage{paralist}
\usepackage[top=2.50cm, bottom=2.50cm, left=2.50cm, right=2.50cm]{geometry}

\usepackage{color}

\numberwithin{equation}{section}

\usepackage{xspace}
\usepackage[bookmarksnumbered,colorlinks]{hyperref}
\usepackage{graphics}
\def\bb#1\eb{\textcolor{blue}
{#1}} %
\def\br#1\er{\textcolor{red}
{#1}} %
\def\bv#1\ev{\textcolor{green}
{#1}} %
\def\bc#1\ec{\textcolor{cyan}
{#1}} %

\usepackage{graphics}

\def\Xint#1{\mathchoice
  {\XXint\displaystyle\textstyle{#1}}%
  {\XXint\textstyle\scriptstyle{#1}}%
  {\XXint\scriptstyle\scriptscriptstyle{#1}}%
  {\XXint\scriptscriptstyle\scriptscriptstyle{#1}}%
  \!\int}
\def\XXint#1#2#3{{\setbox0=\hbox{$#1{#2#3}{\int}$}
  \vcenter{\hbox{$#2#3$}}\kern-.5\wd0}}
\def\-int{\Xint -}

\newcommand{\R}{\mathbb{R}}
\newcommand{\N}{\mathbb{N}}

\DeclareMathOperator{\supp}{supp}
\DeclareMathOperator{\B}{\mathcal{B}}
\DeclareMathOperator{\K}{\mathcal{K}}
\DeclareMathOperator{\X}{\mathbb{X}}

\newtheorem{prop}{Proposition}[section]
\newtheorem{lem}{Lemma}[section]
\newtheorem{thm}{Theorem}[section]

\newtheorem{remark}{Remark}[section]

\begin{document}
\title[Sign-changing solutions]{Sign-changing solutions for a class of zero mass nonlocal Schr\"odinger equations}

\author[V. Ambrosio]{Vincenzo Ambrosio}
\address{Vincenzo Ambrosio\hfill\break\indent 
Department of Mathematics  \hfill\break\indent
EPFL SB CAMA \hfill\break\indent
Station 8 CH-1015 Lausanne, Switzerland}
\email{vincenzo.ambrosio@uniurb.it}

\author[G. M. Figueiredo]{Giovany M. Figueiredo}
\address{Departamento de Matematica \hfill\break\indent
Universidade de Brasilia-UNB \hfill\break\indent
CEP:70910-900 Brasilia, DF \hfill\break\indent
Brazil}
\email{giovany@ufpa.br}

\author[T. Isernia]{Teresa Isernia}
\address{Teresa Isernia\hfill\break\indent
Dipartimento di Ingegneria Industriale e Scienze Matematiche \hfill\break\indent
Universit\`a Politecnica delle Marche\hfill\break\indent
Via Brecce Bianche, 12\hfill\break\indent
60131 Ancona (Italy)}
\email{teresa.isernia@unina.it}

\author[G. Molica Bisci]{Giovanni Molica Bisci}
\address{Dipartimento P.A.U., \hfill\break\indent
Universit\`a  degli Studi Mediterranea di Reggio Calabria, \hfill\break\indent
Salita Melissari - Feo di
Vito, 89100 Reggio Calabria, Italy}
\email{gmolica@unirc.it}


\date{}

\date{}

\begin{abstract}
We consider the following class of fractional Schr\"odinger equations
$$
(-\Delta)^{\alpha} u + V(x)u = K(x) f(u) \mbox{ in } \R^{N}
$$
where $\alpha\in (0, 1)$, $N>2\alpha$, $(-\Delta)^{\alpha}$ is the fractional Laplacian, $V$ and $K$ are positive continuous functions which vanish at infinity, and $f$ is a continuous function.
By using a minimization argument and a quantitative deformation lemma, we obtain the existence of a sign-changing solution. Furthermore, when $f$ is odd, we prove that the above problem admits infinitely many nontrivial solutions. Our result extends to the fractional framework some well-known theorems proved for elliptic equations in the classical setting. With respect to these cases studied in the literature, the nonlocal one considered here presents some additional difficulties,
such as the lack of decompositions involving positive and negative parts, and the non-differentiability of the Nehari Manifold, so that a careful analysis of the fractional spaces involved is necessary.
\end{abstract}

\maketitle

\section{Introduction}

\noindent

A very interesting area of nonlinear analysis lies in the
study of elliptic equations involving fractional operators. Recently,
great attention has been focused on these problems, both for the pure
mathematical research and in view of concrete real-world applications.
Indeed, this type of operators appears in a quite natural way in different
contexts, such as the description of several physical phenomena (see \cite{BV, DPV, MRS}).

In this paper we study the existence of least energy sign-changing (or nodal) solutions for the nonlinear problem involving the fractional Laplacian
\begin{equation}\label{P}
\left\{
\begin{array}{ll}
(-\Delta)^{\alpha} u + V(x)u = K(x) f(u) &\mbox{ in } \R^{N} \\
u\in \mathcal{D}^{\alpha, 2}(\R^{N})
\end{array}
\right.
\end{equation}
with $\alpha\in (0,1)$, $N>2\alpha$ and $\mathcal{D}^{\alpha, 2}(\R^{N})$ being defined as the completion of $u\in C^{\infty}_{c}(\R^{N})$ with respect to the Gagliardo semi-norm
$$
[u]:=\left(\iint_{\R^{2N}} \frac{|u(x)-u(y)|^{2}}{|x-y|^{N+2\alpha}} dx dy\right)^{1/2},
$$
where $C^{\infty}_{c}(\R^{N})$ is the space of smooth functions with compact support. \par
\indent The operator $(-\Delta)^{\alpha}$ is the so-called fractional Laplacian which, up to a positive constant, may be defined through the kernel representation
$$
(-\Delta)^{\alpha}u(x):={\rm P.V.}\int_{\R^{N}} \frac{u(x)-u(y)}{|x-y|^{N+2\alpha}} dy,
$$
for every $u:\R^{N}\rightarrow \R$ sufficiently smooth. Here, the abbreviation P.V. stands for ``in the principal value sense".\par
For an elementary introduction to the fractional Laplacian and fractional Sobolev spaces we refer the interested reader to \cite{DPV, MRS}.\par
\indent Equation (\ref{P}) appears in the study of standing wave solutions $\psi(x, t)=u(x) e^{-\imath \omega t}$
to the following fractional Schr\"odinger equation
$$
\imath \hbar \frac{\partial\psi}{\partial t}=\hbar^{2}(-\Delta)^{\alpha}\psi+W(x) \psi - f(|\psi|) \mbox{ in } \R^{N}
$$
where $\hbar$ is the Planck constant, $W:\R^{N}\rightarrow \R$ is an external potential and $f$ a suitable nonlinearity.
The fractional Schr\"odinger equation is one of the most important objects of the fractional quantum mechanics because it appears in problems involving nonlinear optics, plasma physics and condensed matter physics.
The previous equation has been introduced for the first time by Laskin \cite{Laskin1, Laskin2} as a result of expanding the Feynman path integral, from the Brownian-like to the L\'evy-like quantum mechanical paths.
To be physically consistent, one of the main features of the Schr\"odinger equation is that, in the semi-classical limit in which the diffusion operator arises as a singular perturbation, the wave functions ``concentrate" into ``particles"; see \cite{DDPW, DDPDV} for more details.\par
Subsequently, many papers appeared studying existence, multiplicity, and behavior of solutions to fractional Schr\"odinger equations \cite{AM, A0, A1, A2, AmbFig, AH, AI2, DiPV} as well as \cite{FQT, FS, FLS, Isernia, MR, Secchi1}.
Generally, nonlinear problems involving nonlocal operators have received a great interest  from the mathematical community thanks to their intriguing structure and in view of several applications, such as, phase transition, optimization, obstacle problems, minimal surfaces and regularity theory; see, among others, the papers \cite{Cab, CRS, CSS, CS, ming1}. 

\indent
When $\alpha=1$, the equation in (\ref{P}) becomes the classical nonlinear Schr\"odinger equation
\begin{equation}\label{CSE}
-\Delta u+V(x) u=K(x) f(u) \mbox{ in } \R^{N},
\end{equation}
which has been extensively studied in the last twenty years. We do not intend to review the huge bibliography of equations like \eqref{CSE},
we just emphasize that the potential $V:\mathbb{R}^n\to \mathbb{R}$ has
a crucial role concerning the existence and behavior of
solutions. For instance, when $V$ is a positive constant,
 or $V$ is radially symmetric, it is natural to look
for radially symmetric solutions,
see  \cite{S,Willem}.
On the other hand, after the seminal paper of
Rabinowitz \cite{seminalrabinowitz}, where the potential $V$
is assumed to be coercive, several different assumptions are adopted in order to obtain
existence and multiplicity results, see \cite{BLW, BPW, GR}.
An important class of problems associated with (\ref{CSE}) is the zero mass case, which occurs when
$$
\lim_{|x|\rightarrow +\infty} V(x)=0.
$$
\indent To study these problems, many authors used several variational methods; see \cite{AS1, AFM, AW, BF, BGM2} and \cite{BL1, BVS, FJ} for problems set in $\R^{N}$, as well as \cite{AS2, BW, BWW, CCN} for problems in bounded domain with homogeneous boundary conditions.\par
\indent
We notice that there is a huge literature on these classical topics, and, in recent years, also a lot of papers related to the study of fractional and nonlocal operators of elliptic type, through critical point theory, appeared. Indeed, a natural question is whether or not these techniques may be adapted in order to investigate the fractional analogue of the classical elliptic case. 
We refer to the recent books \cite{DiMV, MRS}, in which the analysis of some fractional elliptic problems, via classical variational methods and other novel approaches (see, for instance, \cite{AP, FPS, pu1,pu2}) is performed. \\
In this spirit, the goal of the present paper is to study the nonlocal counterpart of (\ref{CSE}), and to prove the existence of sign-changing solutions to problem (\ref{P}).\\
Before stating our main result, we introduce the basic assumptions on $V, K$ and $f$.\par

\indent More precisely, we suppose that the functions $V, K : \R^{N} \rightarrow \R$ are continuous on $\R^{N}$, and we say that $(V, K)\in \K$ if the following conditions hold:
\begin{compactenum}[($h_1$)]
\item $V(x), K(x)>0$ \textit{for all} $x \in \R^{N}$ \textit{and} $K \in L^{\infty}(\R^{N})$;
\item \textit{If $\{A_{n}\}_{n\in \N}\subset \R^{N}$ is a sequence of Borel sets such that the Lebesgue measure $m(A_{n})$ is less than or equal to $R$, for all $n\in \N$ and some $R>0$, then
\begin{equation*}
\lim_{r\rightarrow +\infty} \int_{A_{n}\cap \B_{r}^{c}(0)} K(x) \, dx =0, \,
\end{equation*}
uniformly in $n\in \N$, where $\B_{r}^{c}(0):=\R^N\setminus \B_{r}(0)$ and
$$
\B_{r}(0):=\{x\in\R^N:|x|<r\}.
$$
}
\end{compactenum}

\indent Furthermore, one of the following conditions occurs
\begin{compactenum}[($h_3$)]
\item $\displaystyle{{K}/{V}\in L^{\infty}(\R^{N}) }$
\end{compactenum}
or
\begin{compactenum}[($h_4$)]
\item \textit{there exists $m\in (2, 2^{*}_{\alpha})$ such that}
\begin{equation*}
\frac{K(x)}{V(x)^{\frac{2^{*}_{\alpha}-m}{(2^{*}_{\alpha}-2)}}} \rightarrow 0 \, \mbox{\textit{ as} } |x|\rightarrow +\infty,
\end{equation*}
\textit{where} $\displaystyle{2^{*}_{\alpha}:= \frac{2N}{N-2\alpha}}$.
\end{compactenum}

\smallskip

We recall that the assumptions $(h_{1})$-$(h_{4})$ were introduced for the first time by Alves and Souto in \cite{AS1}. 
It is very important to observe that $(h_2)$ is weaker than any one of the conditions below used in the above-mentioned papers to study zero-mass problems:
\begin{compactenum}[$(a)$]
\item there are $r\geq 1$ and $\rho\geq 0$ such that $K\in L^{r}(\R^{N}\setminus \B_{\rho}(0))$;
\item $K(x)\rightarrow 0$ as $|x|\rightarrow \infty$;
\item $K=H_1+H_2$, with $H_1$ and $H_2$ verifying $(a)$ and $(b)$ respectively.
\end{compactenum}
Now, we provide some examples of functions $V$ and $K$ satisfying $(h_1)$-$(h_4)$. Let $\{B_{n}\}_{n\in \mathbb{N}}$ be a disjoint sequence of open balls in $\R^N$ centered in $\xi_{n}=(n, 0,\dots, 0)$ and consider a nonnegative function $H_3$ such that
$$
H_3=0 \mbox{ in } \R^{N}\setminus \bigcup_{n=1}^{\infty} \B_{n}, \quad H_{3}(\xi_{n})=1\, \mbox{ and } \,\int_{\B_{n}} H_{3}(x)dx=2^{-n}.
$$
Then, the pairs $(V, K)$ given by 
$$
K(x)=V(x)=H_{3}(x)+\frac{1}{\log(2+|x|)}
$$ 
and 
$$
K(x)=H_{3}(x)+\frac{1}{\log(2+|x|)} \quad \mbox{ and } \quad V(x)=H_{3}(x)+\left(\frac{1}{\log(2+|x|)}\right)^{\frac{2^{*}_{\alpha}-2}{2^{*}_{\alpha}-m}}
$$
for some $m\in (2, 2^{*}_{\alpha})$, belong to the class $\mathcal{K}$.

For the nonlinearity $f: \R\rightarrow \R$, we assume that it is a $C^{0}$-function and satisfies the following growth conditions in the origin and at infinity:
\begin{compactenum}[($f_1$)]
\item $\displaystyle{\lim_{|t|\rightarrow 0} \frac{f(t)}{|t|}=0} \mbox{ \textit{if }} (h_3) \mbox{ \textit{holds}}$ or
\end{compactenum}
\begin{compactenum}[($\tilde{f}_1$)]
\item $\displaystyle{\lim_{|t|\rightarrow 0} \frac{f(t)}{|t|^{m-1}}<+\infty} \mbox{ \textit{if} } (h_4)$ \textit{holds.}

\item [($f_{2}$)] $f$ \textit{has a quasicritical growth at infinity, namely}
\begin{equation*}
\lim_{|t|\rightarrow +\infty} \frac{f(t)}{|t|^{2^{*}_{\alpha}-1}}=0;
\end{equation*}

\item [($f_{3}$)] $F$ \textit{has a superquadratic growth at infinity, that is}
\begin{equation*}
\lim_{|t|\rightarrow +\infty} \frac{F(t)}{|t|^{2}}=+\infty,
\end{equation*}
\textit{where, as usual, we set} $\displaystyle{F(t):=\int_{0}^{t} f(\tau) d\tau}$;

\item [($f_{4}$)] The map $\displaystyle{t \mapsto \frac{f(t)}{|t|}}$ is strictly increasing for every $t\in \R\setminus\{0\}$.

\end{compactenum}


\smallskip

As models for $f$ we can take, for instance, the following nonlinearities
\begin{equation*}
f(t)= (t^{+}) ^{m} \quad \mbox{ and } \quad
f(t)= 
\left\{
\begin{array}{ll}
\log 2(t^{+})^{m} &\mbox{ if } t\leq 1\\
t \log(1+ t) & \mbox{ if } t>1,
\end{array}
\right.
\end{equation*}
for some $m\in (2, 2^{*}_{\alpha})$. 
\begin{remark}\label{rem1}\rm{
We notice that by $(f_4)$ it follows that the real function
\begin{align}\begin{split} \label{4.23}
t\mapsto \frac{1}{2} f(t)t - F(t) 
\end{split}\end{align}}
is strictly increasing for every $t>0$ and strictly decreasing for every $t<0$.
\end{remark}

Now, we are ready to state the main result of this paper. 
\begin{thm}\label{thmf}
Suppose that $(V, K)\in \K$ and $f\in C^{0}(\R, \R)$ verifies either $(f_1)$ or $(\tilde{f}_1)$ and $(f_2)-(f_4)$. Then, problem \eqref{P} possesses a least energy nodal weak solution. In addition, if the nonlinear term $f$ is odd, then problem \eqref{P} has infinitely many nontrivial weak solutions not necessarily nodals.
\end{thm}
The proof of Theorem \ref{thmf} is obtained by exploiting variational arguments. One of the main difficulties in the study of problem \eqref{P} is related to the presence of the fractional Laplacian $(-\Delta)^{\alpha}$ which is a nonlocal operator. Indeed, the Euler-Lagrange functional associated to the problem \eqref{P}, that is 
\begin{equation*}
J(u):= \frac{1}{2}\left([u]^{2} + \int_{\R^{N}} V(x)u^{2}\, dx \right) - \int_{\R^{N}} K(x) F(u) \, dx,
\end{equation*}
does not satisfy the decompositions (see Section \ref{tec1})
\begin{align*}
&\langle J'(u), u^{\pm} \rangle= \langle J'(u^{\pm}), u^{\pm} \rangle \\
&J(u)= J(u^{+})+ J(u^{-}),
\end{align*}
which were fundamental in the application of variational methods to study (\ref{CSE}); see \cite{BF, BWW, SW, Willem}. \\
\indent Anyway, along the paper we prove that the geometry of the classical minimization theorem is respected in the nonlocal framework: more precisely, we develop a functional analytical setting that is inspired by (but not equivalent to) the fractional Sobolev spaces, in order to correctly encode the variational formulation of problem (\ref{P}); see Section \ref{prel}. Secondly, the nonlinearity $f$ is only continuous, so to overcome the nondifferentiability of the Nehari manifold associated to $J$, we adapt in our framework, some ideas developed in \cite{SW}. Of course, also the compactness properties (see Proposition \ref{propz2}) required by these abstract theorems are satisfied in the nonlocal case, thanks to our functional setting. \\
\indent Then, in order to obtain nodal solutions, we look for critical points of $J(tu^{+}+su^{-})$, and, due to the fact that $f$ is only continuous, we do not apply the Miranda's theorem \cite{Miranda} as in \cite{AS2, BF}, but we use an iterative procedure and the properties of $J$ to prove the existence of a sequence which converges to a critical point of $J(tu^{+}+su^{-})$ (see Lemma \ref{lemz11}).\\
\indent Finally, we emphasize that Theorem \ref{thmf} improves the recent result established in \cite{AI}, in which the existence of a least energy nodal solution to problem (\ref{P}) has been proved under the stronger assumption that $f\in C^{1}$ and satisfies the Ambrosetti-Rabinowitz condition.

\indent
The paper is organized as follows. In Section \ref{prel} we present the variational framework of the problem and compactness results which will be useful for the next sections. In Section \ref{ene}, we obtain some preliminary results which are useful to overcome the lack of differentiability of  the Nehari manifold in which we look for weak solutions to problem (\ref{P}).
In Section \ref{tec1} we provide the proofs of some technical lemmas. Finally, in Section \ref{fine} we prove the existence of a least energy nodal weak solution by using minimization arguments and a variant of the Deformation Lemma.

\section{Preliminary results}\label{prel}

\noindent
This section is devoted to the notations used along the present paper. In order to give the weak formulation of problem~\eqref{P}, we need to work in a special functional space. Indeed, one of the difficulties in treating problem~\eqref{P} is related to his variational formulation. With respect to this, the standard fractional Sobolev spaces are not sufficient in order to study the problem. We overcome this difficulty by working in a suitable functional space, whose definition and basic analytical properties are recalled here.\par
For $\alpha\in (0, 1)$, we denote by $\mathcal{D}^{\alpha, 2}(\R^{N})$ the completion of $C^{\infty}_{c}(\R^{N})$ with respect to the so called Gagliardo semi-norm
$$
[u]:=\left(\iint_{\R^{2N}} \frac{|u(x)-u(y)|^{2}}{|x-y|^{N+2\alpha}} dx dy\right)^{1/2}
$$
and $H^{\alpha}(\R^{N})$ denotes the standard fractional Sobolev space, defined as the set of $u\in \mathcal{D}^{\alpha, 2}(\R^{N})$ satisfying $u\in L^{2}(\R^{N})$ with the norm
$$
\|u\|_{H^{\alpha}(\R^{N})}:=\left([u]^{2}+\|u\|^{2}_{L^{2}(\R^{N})}\right)^{1/2}.
$$
\indent Let us introduce the following functional space
\begin{equation*}
\X:= \left\{u \in \mathcal{D}^{\alpha, 2}(\R^{N}) \, : \, \int_{\R^{N}} V(x) |u|^{2} \, dx <+\infty \right\}
\end{equation*}
endowed with the norm
\begin{equation*}
\|u\|:= \left([u]^{2} + \int_{\R^{N}} V(x)|u|^{2} dx\right)^{1/2}.
\end{equation*}

\indent
At this point, we define, for $q \in \R$ with $q\geq 1$, the Lebesgue space $L^{q}_{K}(\R^{N})$ as
\begin{equation*}
L_{K}^{q} (\R^{N}) := \left \{ u: \R^{N} \rightarrow \R \, \mbox{ measurable and } \int_{\R^{N}} K(x) |u|^{q} \, dx < \infty \right \},
\end{equation*}
endowed with the norm
\begin{equation*}
\|u \|_{L_{K}^{q} (\R^{N})}:= \left( \int_{\R^{N}} K(x) |u|^{q} \, dx \right)^{1/q}.
\end{equation*}

\indent
Finally, we recall the following useful results, which extend the ones in \cite{AS1}.
\begin{lem}\label{cont}
Assume that $(V, K)\in \mathcal{K}$. Then $\X$ is continuously embedded in $L^{q}_{K}(\R^{N})$ for every $q\in [2, 2^{*}_{\alpha}]$ if $(h_3)$ holds. Moreover, $\X$ is continuously embedded in $L^{m}_{K}(\R^{N})$ if $(h_4)$ holds.
\end{lem}

\begin{lem}\label{prop2.2}
Assume that $(V, K)\in \K$. The following facts hold:
\begin{itemize}
\item [{$(1)$}] $\X$ is compactly embedded into $L_{K}^{q}(\R^{N})$ for all $q \in (2, 2^{*}_{\alpha})$ if $(h_3)$ holds$;$
\item [{$(2)$}] $\X$ is compactly embedded into $L_{K}^{m}(\R^{N})$ if $(h_4)$ holds.
\end{itemize}
\end{lem}

\begin{lem}\label{lem2.1}
Assume that $(V, K)\in \K$ and $f$ satisfies either $(f_1)-(f_2)$ or $(\tilde{f}_1)-(f_2)$. Let $\{u_{n}\}_{n\in \N}$ be a sequence such that $u_{n}\rightharpoonup u$ in $\X$. Then, up to a subsequence, one has
\begin{equation*}
\lim_{n\rightarrow \infty} \int_{\R^{N}} K(x) F(u_{n}) \, dx = \int_{\R^{N}} K(x) F(u) \, dx
\end{equation*}
and
\begin{equation*}
\lim_{n\rightarrow \infty} \int_{\R^{N}} K(x) f(u_{n}) u_{n} \, dx = \int_{\R^{N}} K(x) f(u) u \, dx.
\end{equation*}
\end{lem}

For a proof of these lemmas one can see \cite{AI}.

\section{Existence of a least energy nodal solution}\label{ene}

In this section we obtain some preliminary results which are useful to overcome the lack of differentiability of  the Nehari manifold in which we look for weak solutions to problem (\ref{P}).

\indent
In the following we search a nodal or sign-changing weak solution of problem (\ref{P}), that is a function $u=u^{+}+u^{-}\in \X$ such that $u^{+}:=\max \{u, 0\}\neq 0$, $u^{-}:= \min \{u, 0\}\neq 0$ in $\R^{N}$ and
\begin{equation*}
\iint_{\R^{2N}} \frac{(u(x)- u(y)) (\varphi(x)-\varphi(y))}{|x-y|^{N+2\alpha}} \, dxdy + \int_{\R^{N}} V(x) u(x) \varphi(x) \, dx
\end{equation*}
$$
= \int_{\R^{N}} K(x) f(u) \varphi(x) \, dx,
$$
for every $\varphi \in \X$. \par

The energy functional associated to problem (\ref{P}) is given by
\begin{equation*}
J(u):= \frac{1}{2}\|u\|^{2} - \int_{\R^{N}} K(x) F(u) \, dx,
\end{equation*}
where $\displaystyle{F(t):= \int_{0}^{t} f(\tau) d\tau}$.
By the assumptions on $f$, it is clear that $J\in C^{1}(\X, \R)$ and that its differential $J': \X \rightarrow \X'$ is given by
\begin{equation*}
\langle J'(u), \varphi \rangle = \iint_{\R^{2N}} \frac{(u(x)- u(y)) (\varphi(x)-\varphi(y))}{|x-y|^{N+2\alpha}} \, dx + \int_{\R^{N}} V(x) u(x) \varphi(x) \, dx
\end{equation*}
$$
- \int_{\R^{N}} K(x) f(u) \varphi(x) \, dx,
$$
for every $u, \varphi \in \X$.
Then, the critical points of $J$ are the weak solutions of problem (\ref{P}).\\
\indent Let us also observe that $J$ satisfies the following decompositions
\begin{align*}
J(u)=J(u^{+})+J(u^{-}) -\iint_{\R^{2N}} \frac{u^{+}(x)u^{-}(y) + u^{-}(x) u^{+}(y)}{|x-y|^{N+2\alpha}} \, dxdy,
\end{align*}
and
\begin{align*}
\langle J'(u), u^{+}\rangle =\langle J'(u^{+}), u^{+}\rangle -\iint_{\R^{2N}} \frac{u^{+}(x)u^{-}(y) + u^{-}(x) u^{+}(y)}{|x-y|^{N+2\alpha}} \, dxdy.
\end{align*}

\indent
The Nehari manifold associated to the functional $J$ is given by
\begin{equation*}
\mathcal{N}:= \{ u\in \X \setminus \{0\} : \langle J'(u), u \rangle =0\}.
\end{equation*}
Recalling that a nonzero critical point $u$ of $J$ is a least energy weak solution of problem (\ref{P}) if
$$
J(u)=\min_{v\in \mathcal{N}} J(v)
$$
and, since our purpose is to prove the existence of a least energy sign-changing weak solution of (\ref{P}), we look for $u\in \mathcal{M}$ such that
$$
J(u)=\min_{v\in \mathcal{M}} J(v)
$$
where $\mathcal{M}$ is the subset of $\mathcal{N}$ containing all sign-changing weak solutions of problem (\ref{P}), that is
\begin{equation*}
\mathcal{M}:= \{ w\in \mathcal{N} : w^{+}\neq 0, w^{-}\neq 0,  \langle J'(w), w^{+} \rangle = \langle J'(w), w^{-} \rangle=0\}.
\end{equation*}

Once $f$ is only continuous, the following results are crucial, since they allow us to overcome the non-differentiability of $\mathcal{N}$. Below, we denote by $\mathbb{S}$  the unit sphere on $\X$.

\begin{lem}\label{lemz1}
Suppose that $(V, K)\in \mathcal{K}$ and $f$ verifies conditions $(f_1)-(f_4)$. Then, the following facts hold true:
\begin{compactenum}
\item[$(a)$] For each $u\in \X\setminus\{0\}$, let $h_{u}: \R_{+} \rightarrow \R$ be defined by $h_{u}(t):= J(tu)$. Then, there is a unique $t_{u}>0$ such that
\begin{align*}
&h_{u}'(t)>0 \, \mbox{ in } (0, t_{u})\\
&h_{u}'(t)<0 \, \mbox{ in } (t_{u}, +\infty);
\end{align*}
\item[$(b)$] There is $\tau>0$, independent of $u$, such that $t_{u}\geq \tau$ for every $u\in \mathbb{S}$. Moreover, for each compact set $\mathcal{W}\subset \mathbb{S}$, there is $C_{\mathcal{W}}>0$ such that $t_{u}\leq C_{\mathcal{W}}$ for every $u\in \mathcal{W}$$;$
\item[$(c)$] The map $\hat{\eta}: \X\setminus \{0\}\rightarrow \mathcal{N}$ given by $\hat{\eta}(u):=t_{u}u$ is continuous and $\eta:= \hat{\eta}|_{\mathbb{S}}$ is a homeomorphism between $\mathbb{S}$ and $\mathcal{N}$. Moreover,
    $$\eta^{-1}(u)= \frac{u}{\|u\|}.$$
\end{compactenum}
\end{lem}

\begin{proof}
$(a)$ We distinguish two cases. \par
\indent Let us assume that $(h_3)$ is verified. By using assumptions $(f_1)$ and $(f_2)$, given $\varepsilon>0$ there exists a positive constant $C_{\varepsilon}$ such that
\begin{equation*}
|F(t)| \leq \varepsilon t^{2} + C_{\varepsilon} |t|^{2^{*}_{\alpha}},\quad\mbox{ for every } t\in\R.
\end{equation*}
The above inequality and the Sobolev embedding yield
\begin{align}\begin{split}\label{z1}
J(tu)&= \frac{t^{2}}{2}[u]^{2} + \frac{t^{2}}{2}\int_{\R^{N}} V(x) |u|^{2} \, dx - \int_{\R^{N}} K(x) F(tu) \, dx \\
&\geq \frac{t^{2}}{2} \|u\|^{2} - \varepsilon \int_{\R^{N}} K(x) t^{2} u^{2}\, dx - C_{\varepsilon} \int_{\R^{N}} K(x) t^{2^{*}_{\alpha}} |u|^{2^{*}_{\alpha}}\, dx \\
&\geq \frac{t^{2}}{2} \|u\|^{2} - \varepsilon \left\|{K}/{V}\right\|_{L^{\infty}(\R^{N})} t^{2} \|u\|^{2} - C_{\varepsilon}C' \|K\|_{L^{\infty}(\R^{N})} t^{2^{*}_{\alpha}} \|u\|^{2^{*}_{\alpha}}.
\end{split}\end{align}

\noindent Taking $\displaystyle{0<\varepsilon< \frac{1}{2 \displaystyle\left\|{K}/{V}\right\|_{L^{\infty}(\R^{N})}}}$ we get $t_{0}>0$ sufficiently small such that
\begin{equation}\label{z2}
0<h_{u}(t)= J(tu), \, \mbox{ for all } t<t_{0}.
\end{equation}
\indent On the other hand, suppose that $(h_4)$ is true. Then, there exists a constant $C_{m}>0$ such that, for each $\varepsilon \in (0, C_{m})$, we obtain $R>0$ such that
\begin{equation}\label{z3}
\int_{\B_{R}^{c}(0)} K(x) |u|^{m} \, dx \leq \varepsilon \int_{\B_{R}^{c}(0)} (V(x) |u|^{2} + |u|^{2^{*}_{\alpha}}) \, dx,
\end{equation}
for every $u\in \X$.\par
\indent Now, by using $(\tilde{f}_1)$ and $(f_2)$,
the Sobolev embedding result, relation (\ref{z3}) and H\"older inequality, we have that
\begin{align}\label{z11}
J(tu) &\geq \frac{t^{2}}{2}\|u\|^{2} - C_{1} \int_{\R^{N}} K(x) t^{m} |u|^{m} \, dx -C_{2} \int_{\R^{N}} K(x) t^{2^{*}_{\alpha}} |u|^{2^{*}_{\alpha}} \, dx \nonumber \\
&\geq \frac{t^{2}}{2}\|u\|^{2} - C_{1} t^{m} \varepsilon \int_{\B_{R}^{c}(0)} (V(x) |u|^{2} + |u|^{2^{*}_{\alpha}}) \, dx - C_{1} t^{m} \int_{\B_{R}(0)} K(x) |u|^{m}\, dx \nonumber \\
&\,\,\,\,\,- C_{2} t^{2^{*}_{\alpha}} \|K\|_{L^{\infty}(\R^{N})} \int_{\R^{N}} |u|^{2^{*}_{\alpha}} \, dx \nonumber \\
&\geq \frac{t^{2}}{2}\|u\|^{2} - C_{1} t^{m} \varepsilon \int_{\B_{R}^{c}(0)} (V(x) |u|^{2} + |u|^{2^{*}_{\alpha}}) \, dx - C_{1} t^{m} \|K\|_{L^{\frac{2^{*}_{\alpha}}{2^{*}_{\alpha}-m}}(\B_{R}(0))} \left( \int_{\B_{R}(0)} K(x) |u|^{m}\, dx \right)^{\frac{m}{2^{*}_{\alpha}}} \nonumber\\
&\,\,\,\,\,- C_{2} t^{2^{*}_{\alpha}} \|K\|_{L^{\infty}(\R^{N})} \int_{\R^{N}} |u|^{2^{*}_{\alpha}} \, dx \nonumber \\
&\geq \frac{t^{2}}{2}\|u\|^{2} - C_{1} t^{m} \left( \varepsilon \|u\|^{2} + \varepsilon C \|u\|^{2^{*}_{\alpha}} + C \|K\|_{L^{\frac{2^{*}_{\alpha}}{2^{*}_{\alpha}-m}}(\B_{R}(0))} \|u\|^{m} \right) - C_{2} C t^{2^{*}_{\alpha}} \|K\|_{L^{\infty}(\R^{N})} \|u\|^{2^{*}_{\alpha}}.
\end{align}
This shows that condition (\ref{z2}) is verified also in this case. \par
\indent Moreover, since $F(t)\geq 0$ for every $t\in \R$, we have
\begin{equation*}
J(tu)\leq \frac{t^{2}}{2} \|u\|^{2} - \int_{A} K(x) F(tu)\, dx,
\end{equation*}
where $A\subset \supp u$ is a measurable set with finite and positive measure. Hence,
\begin{equation*}
\limsup_{t\rightarrow +\infty} \frac{J(tu)}{\|tu\|^{2}} \leq \frac{1}{2} -\liminf_{t\rightarrow \infty} \left \{\int_{A} K(x)\left[ \frac{F(tu)}{(tu)^{2}} \right] \left(\frac{u}{\|u\|}\right)^{2} \, dx\right\}.
\end{equation*}
\indent By $(f_3)$ and Fatou's lemma it follows that
\begin{equation}\label{z4}
\limsup_{t\rightarrow +\infty} \frac{J(tu)}{\|tu\|^{2}}\leq -\infty.
\end{equation}
\noindent Thus, there exists $R>0$ sufficiently large such that
\begin{equation}\label{z5}
h_{u}(R)= J(Ru)<0.
\end{equation}
By the continuity of $h_{u}$ and $(f_4)$ there is $t_{u}>0$ which is a global maximum of $h_{u}$ with $t_{u}u \in \mathcal{N}$. \par
\indent Now, we aim to prove that $t_{u}$ is the unique critical point of $h_u$. Arguing by contradiction, let us assume that there are $t_1, t_2$ critical points of $h_u$ with $t_1>t_2>0$. Thus, we have $h_{u}'(t_{1})=h_{u}'(t_{2})=0$, or equivalently
\begin{align*}
&\|u\|^{2}- \int_{\R^{N}} K(x) \frac{f(t_{1}u)u}{t_{1}}\, dx=0\\
&\|u\|^{2}- \int_{\R^{N}} K(x) \frac{f(t_{2}u)u}{t_{2}}\, dx=0.
\end{align*}
\indent Subtracting, and taking into account $(f_{4})$, we obtain
\begin{equation*}
0 = \int_{\R^{N}} K(x) \left[ \frac{f(t_{1}u)}{t_{1}u} - \frac{f(t_{2}u)}{t_{2}u} \right] u^{2} \, dx>0
\end{equation*}
which leads a contradiction.
\smallskip

\noindent
$(b)$ By $(a)$ there exists $t_{u}>0$ such that
\begin{equation}\label{t_u}
t_{u}^{2} \|u\|^{2} = \int_{\R^{N}} K(x) f(t_{u}u) t_{u}u\, dx.
\end{equation}
\indent Then, estimating the right-hand side of (\ref{t_u}) similarly to (\ref{z1}) and (\ref{z11}), we obtain that there exists $\tau>0$, independent of $u$, such that $t_{u}\geq \tau$. 
On the other hand, let $\mathcal{W}\subset \mathbb{S}$ be a compact set. Assume by contradiction that there exists $\{u_{n}\}_{n\in \N}\subset \mathcal{W}$ such that $t_{n}:=t_{u_{n}}\rightarrow \infty$. Therefore, there exists $u\in \mathcal{W}$ such that $u_{n}\rightarrow u$ in $\X$. From (\ref{z4}), we have
\begin{equation}\label{z6}
J(t_{n}u_{n})\rightarrow -\infty \, \mbox{ in } \R.
\end{equation}
\indent Therefore, by using Remark \ref{rem1}, one has
\begin{align}\begin{split}\label{z7}
J(v)&= J(v) - \frac{1}{2} \langle J'(v), v\rangle \\
&= \int_{\R^{N}} K(x) \left[ \frac{1}{2} f(v)v - F(v) \right] \, dx \geq 0,
\end{split}\end{align}
for each $v\in \mathcal{N}$.\par
By taking into account that $\{t_{u_{n}}u_{n}\}_{n\in\N}\subset \mathcal{N}$, we conclude from (\ref{z6}) that (\ref{z7}) is not true, which is a contradiction.
\smallskip

\noindent
$(c)$ Since $J\in C^{1}(\X, \R)$, $J(0)=0$ and since it satisfies $(a)$ and $(b)$, the thesis follows by [\cite{SW}, Proposition 8]. The proof is now complete.
\end{proof}

Let us define the maps
\begin{equation*}
\hat{\psi}: \X \rightarrow \R \, \mbox{ and } \, \psi: \mathbb{S}\rightarrow \R,
\end{equation*}
by $\hat{\psi}(u):= J(\hat{\eta}(u))$ and $\psi:=\hat{\psi}|_{\mathbb{S}}$. \par

The compactness condition assumed in the sequel is the well-known {Palais-Smale condition at level $d$}, (briefly ${(\rm PS)}_{d}$) which in our framework reads as follows (see, for instance, \cite{Rab, Willem}):
\begin{center}
$J$ satisfies the {\em Palais-Smale compactness condition} at level $d\in \R$\\
if any sequence $\{u_n\}_{n\in\N}$ in $\X$ such that\\
$J(u_n)\to d \,\, \mbox{and}\,\, \sup\Big\{
\big|\langle\,J'(u_n),\varphi\,\rangle \big|\,: \;
\varphi\in \X\,,
\|\varphi\|=1\Big\}\to 0$
as $n\to \infty$,\\
admits a strongly convergent subsequence in $\X$\,.
\end{center}


\noindent The next result is a consequence of Lemma \ref{lemz1}.

\begin{prop}\label{propz2}
Suppose that $(V, K)\in \mathcal{K}$ and $f$ verifies $(f_1)-(f_4)$. Then, one has the following assertions:
\begin{compactenum}
\item[$(a)$] $\hat{\psi}\in C^{1}(\X\setminus\{0\}, \R)$ and
\begin{equation*}
\langle \hat{\psi}'(u), v \rangle=\frac{\|\hat{\eta}(u)\|}{\|u\|} \langle J'(\hat{\eta}(u)), v \rangle \,,
\end{equation*}
for every $u\in \X\setminus \{0\}$ and $v\in \X$$;$
\item[$(b)$] $\psi \in C^{1}(\mathbb{S}, \R)$ and $\langle \psi'(u), v \rangle = \|\eta(u)\| \langle J'(\eta(u)), v \rangle$, for every
\begin{equation*}
v\in T_{u}\mathbb{S}:= \left\{v \in \X : \langle v, u \rangle = \int_{\R^{2N}} \frac{(v(x)-v(y))(u(x)-u(y))}{|x-y|^{N+2\alpha}}\, dxdy + \int_{\R^{N}} V(x) uv\, dx=0 \right\};
\end{equation*}
\item[$(c)$] If $\{u_{n}\}_{n\in\N}$ is a ${(\rm PS)}_{d}$ sequence for $\psi$, then $\{\eta(u_{n})\}_{n\in\N}$ is a ${(\rm PS)}_{d}$ sequence for $J$. Moreover, if $\{u_{n}\}_{n\in\N}\subset \mathcal{N}$ is a bounded ${(\rm PS)}_{d}$ sequence for $J$, then $\{\eta^{-1} (u_{n})\}_{n\in \N}$ is a ${(\rm PS)}_{d}$ sequence for the functional $\psi$$;$
\item[$(d)$] $u$ is a critical point of $\psi$ if and only if $\eta(u)$ is a nontrivial critical point for $J$. Moreover, the corresponding critical values coincide and
\begin{equation*}
\inf_{u\in\mathbb{S}} \psi(u) = \inf_{u\in\mathcal{N}} J(u).
\end{equation*}
\end{compactenum}
\end{prop}

\begin{remark}\label{lemz3}\rm{
We notice that the following equalities hold:
\begin{align}\begin{split}\label{z8}
d_{\infty}&:= \inf_{u\in \mathcal{N}} J(u) \\
&= \inf_{u\in \X\setminus \{0\}} \max_{t>0} J(tu) \\
&= \inf_{u\in \mathbb{S}} \max_{t>0} J(tu).
\end{split}\end{align}
In particular, relations (\ref{z1}), (\ref{z4}) and (\ref{z8}) imply that
\begin{equation}\label{z9}
d_{\infty}>0.
\end{equation}}
\end{remark}

\section{Technical lemmas}\label{tec1}

\noindent
The aim of this section is to prove some technical lemmas related to the existence of a least energy nodal solution.\par
For each $u\in \X$ with $u^{\pm}\not\equiv 0$, let us consider the function $h^{u}: [0, +\infty) \times [0, +\infty) \rightarrow \R$ given by
\begin{equation}\label{hv}
h^{u}(t, s) := J(tu^{+} + su^{-}).
\end{equation}
\indent Let us observe that its gradient $\Phi^{u}: [0, +\infty) \times [0, +\infty) \rightarrow \R^{2}$ is defined by
\begin{align}\begin{split}\label{Phiv}
\Phi^{u}(t, s) &:= \left ( \Phi_{1}^{u}(t, s) , \Phi_{2}^{u}(t, s) \right)  \\
&= \left(\frac{\partial h^{u}}{\partial t}(t, s), \frac{\partial h^{u}}{\partial s}(t, s)  \right)  \\
&= \left(\langle J'(tu^{+}+su^{-}), u^{+}\rangle , \langle J'(tu^{+}+su^{-}), u^{-}\rangle  \right).
\end{split}\end{align}

\begin{lem}\label{lemz11}
Suppose that $(V, K)\in \mathcal{K}$ and $f$ verifies $(f_1)-(f_4)$. Then, it follows that
\begin{compactenum}[(i)]
\item[$(i)$] The pair $(t,s)$ is a critical point of $h^{u}$ with $t,s>0$ if, and only if, $tu^{+}+su^{-}\in \mathcal{M}$$;$
\item[$(ii)$] The map $h^{u}$ has a unique critical point $(t_{+}, s_{-})$, with $t_{+}= t_{+}(u)>0$ and $s_{-}= s_{-}(u)>0$, which is the unique global maximum point of $h^{u}$$;$
\item[$(iii)$] The maps $a_{+}(r):= \Phi_{1}^{u}(r, s_{-})r$ and $a_{-}(r):= \Phi_{2}^{u}(t_{+},r)r$ are such that
\begin{align}\begin{split}\label{mapa+-}
& a_{+}(r)>0 \mbox{ if } r\in (0, t_{+}) \,\,\,\,{\rm and}\,\,\,\, a_{+}(r)<0 \mbox{ if } r\in (t_{+}, +\infty) \\
& a_{-}(r)>0 \mbox{ if } r\in (0, s_{-}) \,\,\,\,{\rm and}\,\,\,\,  a_{-}(r)<0 \mbox{ if } r\in (s_{-}, +\infty).
\end{split}\end{align}
\end{compactenum}
\end{lem}

\begin{proof}
$(i)$ Let us observe that by (\ref{Phiv}) we have
\begin{align*}
\Phi^{u}(t, s)=\left(\frac{1}{t} \langle J'(tu^{+}+su^{-}), tu^{+}\rangle, \frac{1}{s} \langle J'(tu^{+}+su^{-}), su^{+}\rangle \right),
\end{align*}
for every $t, s>0$.
Then, $\Phi^{u}(t,s)=0$ if, and only if,
\begin{equation*}
\langle J'(tu^{+}+su^{-}), tu^{+}\rangle =0 \, \mbox{ and } \, \langle J'(tu^{+}+su^{-}), su^{-}\rangle=0,
\end{equation*}
and this implies that $tu^{+}+su^{-} \in \mathcal{M}$.\par
\bigskip
\noindent
$(ii)$ Firstly we show that $h^{u}$ has a critical point. For each $u\in \X$ such that $u^{\pm}\neq 0$ and $s_{0}$ fixed, we define the function $h_{1}: [0, +\infty) \rightarrow [0, +\infty)$ by $h_{1}(t):=h^{u}(t, s_{0})$. Following the lines of Lemma \ref{lemz1}-$(a)$, we can infer that $h_1$ has a maximum positive point.\par
 Moreover, there exists a unique $t_{0}=t_{0}(u, s_{0})>0$ such that
\begin{align*}
&h_{1}'(t)>0 \, \mbox{ if } t\in (0, t_{0}) \\
&h_{1}'(t_{0})=0\\
&h_{1}'(t)<0 \, \mbox{ if } t\in (t_{0}, +\infty).
\end{align*}
Thus, the map $\phi_1: [0, +\infty) \rightarrow [0, +\infty)$ defined by $\phi_1(s):=t(u,s)$, where $t(u,s)$ satisfies the properties just mentioned with $s$ in place of $s_{0}$, is well defined.\par
By the definition of $h_1$ we have
\begin{equation}\label{h'_1}
h_{1}'(\phi_1(s))= \Phi_1^u(\phi_1(s), s)=0 \quad \forall s\geq 0,
\end{equation}
that is
$$
0=|\phi_1(s)|^{2} \|u^{+}\|^{2} - s \phi_1(s) \iint_{\R^{2N}} \frac{u^{+}(x)u^{-}(y) + u^{-}(x) u^{+}(y)}{|x-y|^{N+2\alpha}} \, dxdy
$$
\begin{equation}\label{a4}
 - \int_{\R^{N}} K(x) f(\phi_1(s)u^{+})\, \phi_1(s)u^{+} \, dx.
\end{equation}
\indent Now, we prove some properties of $\phi_1$. \par
\noindent $a)$ The map $\phi_1$ is continuous. \par
Let $s_n\rightarrow s_0$ as $n\rightarrow \infty$ in $\R$. We want to prove that $\{\phi_1(s_n)\}_{n\in \N}$ is bounded. Assume by contradiction that there is a subsequence, again denoted by $\{s_n\}_{n\in\N}$, such that $\phi_1(s_{n})\rightarrow +\infty$ as $n\rightarrow \infty$. So, $\phi_1(s_n)\geq s_n$ for $n$ large. By (\ref{a4}) we have
\begin{equation}\label{a5}
\|u^{+}\|^{2} - \frac{s_{n}}{\phi_1(s_{n})} \iint_{\R^{2N}} \frac{u^{+}(x)u^{-}(y) + u^{-}(x) u^{+}(y)}{|x-y|^{N+2\alpha}} \, dxdy = \int_{\R^{N}} K(x) \frac{f(\phi_1(s_{n})u^{+})}{\phi_1(s_{n})u^{+}}(u^{+})^{2} \, dx.
\end{equation}
\indent Taking into account that $s_n\rightarrow s_0$, $\phi_1(s_{n})\rightarrow +\infty$ as $n\rightarrow \infty$, assumptions $(f_3)-(f_4)$ and Fatou's lemma, yield
\begin{equation*}
\|u^{+}\|^{2} = \liminf_{n\rightarrow \infty} \int_{\R^{N}} K(x) \frac{f(\phi_1(s)u^{+})}{\phi_1(s)u^{+}}(u^{+})^{2} \, dx\geq +\infty.
\end{equation*}
Thus we have a contradiction. So the sequence $\{\phi_1(s_n)\}_{n\in\N}$ is bounded.\par
 Therefore there exists $t_{0}\geq 0$ such that $\phi_1(s_n)\rightarrow t_0$. Consider (\ref{a4}) with $s=s_n$, and by passing to the limit as $n\rightarrow \infty$ we have
\begin{equation*}
t_{0}^{2} \|u^{+}\|^{2}- s_{0} t_{0} \iint_{\R^{2N}} \frac{u^{+}(x)u^{-}(y) + u^{-}(x) u^{+}(y)}{|x-y|^{N+2\alpha}} \, dxdy = \int_{\R^{N}} K(x) f(\phi_1(t_0)u^{+}) \phi_1(t_0)u^{+} \, dx,
\end{equation*}
that is $h_{1}'(t_0)= \Phi^{u}_1(t_{0}, s_{0})=0$. As a consequence, $t_{0}= \phi_1(s_{0})$, i.e. $\phi_1$ is continuous. \\
\smallskip

\noindent
$b)$ $\phi_1(0)>0$. \par
Assume that there exists a sequence $\{s_n\}_{n\in\N}$ such that $\phi_1(s_n)\rightarrow 0^{+}$ and $s_n\rightarrow 0$ as $n\rightarrow \infty$. By assumption $(f_1)$ we get
\begin{align*}
\|u^{+}\|^{2} &\leq \|u^{+}\|^{2} - \frac{s_{n}}{\phi_1(s_n)} \iint_{\R^{2N}} \frac{u^{+}(x)u^{-}(y) + u^{-}(x) u^{+}(y)}{|x-y|^{N+2\alpha}} \, dxdy \\
&= \int_{\R^{N}} K(x) \frac{f(\phi_1(s_n)u^{+})}{\phi_1(s_n)u^{+}}(u^{+})^{2} \, dx \rightarrow 0,\quad \, \mbox{ as } n\rightarrow \infty
\end{align*}
and this fact gives a contradiction. So we deduce that $\phi_1(0)>0$.
\smallskip

\noindent
$c)$ Now we show that $\phi_1(s)\leq s$ for $s$ large. \par
\indent As a matter of fact, proceeding as in the first part of the proof of $a)$, we can see that
it is not possible to find any sequence $\{s_n\}_{n\in\N}$ such that $s_n\rightarrow +\infty$ and $\phi_1(s_n)\geq s_n$ for all $n\in \N$. This implies that $\phi_1(s)\leq s$ for $s$ large.
\bigskip

\indent
Analogously, for every $t_0\geq 0$ we define $h_2(s):= h^{u}(t_{0}, s)$ and, as a consequence, we can find a map $\phi_2$ such that
\begin{equation}\label{h'_2}
h'_{2}(\phi_{2}(t))= \Phi_{2}^{u}(t,\phi_{2}(t)), \quad \forall t\geq 0
\end{equation}
and satisfying $a)$, $b)$ and $c)$.\par
\smallskip
By $c)$ we can find a positive constant $C_1$ such that $\phi_1(s)\leq s$ and $\phi_2(t)\leq t$ for every $t, s \geq C_{1}$. \\
Let
$$
C_{2}:=\max\left\{\max_{s\in [0, C_{1}]} \phi_{1}(s), \max_{t\in [0, C_{1}]} \phi_{2}(t)\right\}
$$
and $C:=\max\{C_{1}, C_{2}\}$.\par
 We define $T: [0, C]\times [0, C]\rightarrow \R^{2}$ by $T(t, s):=(\phi_{1}(s), \phi_{2}(t))$. Let us note that $$T([0, C]\times [0, C])\subset [0, C]\times [0, C].$$
 Indeed, for every $t\in [0, C]$, we have that
\begin{equation*}
\left\{
\begin{array}{ll}
\phi_{2}(t)\leq t\leq C  &\mbox{ if } t\geq C_{1} \\
\phi_{2}(t)\leq \displaystyle\max_{t\in [0, C_{1}]} \phi_{2}(t)\leq C_{2} &\mbox{ if } t\leq C_{1}
\end{array}.
\right.
\end{equation*}
Similarly, we can see that $\phi_{1}(s)\leq C$ for all $s\in [0, C]$. Moreover, since $\phi_{i}$ are continuous for $i=1, 2$, it is clear that $T$ is a continuous map.\par
 Then, by the Brouwer fixed point theorem, there exists $(t_{+}, s_{-})\in [0, C]\times [0, C]$ such that
\begin{equation}
(\phi_{1}(s_{-}), \phi_{2}(t_{+}))= (t_{+}, s_{-}).
\end{equation}
Owing to this fact and recalling that $\phi_{i}>0$, we have $t_{+}>0$ and $s_{-}>0$. By  (\ref{h'_1}) and (\ref{h'_2}) we have
\begin{equation*}
\Phi_{1}^{u}(t_{+}, s_{-})= \Phi_{2}^{u}(t_{+}, s_{-}) =0,
\end{equation*}
that is $(t_{+}, s_{-})$ is a critical point of $h^{u}$.
Next we aim to prove the uniqueness of $(t_{+}, s_{-})$.\par
 Assuming that $w\in \mathcal{M}$, we have
\begin{align*}
\Phi^{w}(1, 1) &= \left ( \Phi_{1}^{w}(1, 1) , \Phi_{2}^{w}(1, 1) \right)  \\
&= \left(\frac{\partial h^{w}}{\partial t}(1, 1), \frac{\partial h^{w}}{\partial s}(1, 1)  \right)  \\
&= \left(\langle J'(w^{+}+w^{-}), w^{+}\rangle , \langle J'(w^{+}+w^{-}), w^{-}\rangle  \right)=(0,0)
\end{align*}
which implies that $(1,1)$ is a critical point of $h^{w}$. Now, assume that $(t_{0}, s_{0})$ is a critical point of $h^{w}$, with $0<t_{0}\leq s_{0}$. This means that
\begin{equation*}
\langle J'(t_{0}w^{+}+ s_{0} w^{-}), t_{0} w^{+} \rangle =0 \, \mbox{ and } \, \langle J'(t_{0}w^{+}+ s_{0} w^{-}), s_{0} w^{-} \rangle =0,
\end{equation*}
or equivalently
\begin{align}
&t_{0}^{2} \|w^{+}\|^{2} - s_{0}t_{0} \iint_{\R^{2N}} \frac{w^{+}(x)w^{-}(y) + w^{-}(x) w^{+}(y)}{|x-y|^{N+2\alpha}} \, dxdy = \int_{\R^{N}} K(x) f(t_0 w^{+}) t_{0}w^{+} \, dx \label{a13} \\
&s_{0}^{2} \|w^{-}\|^{2} - s_{0}t_{0} \iint_{\R^{2N}} \frac{w^{+}(x)w^{-}(y) + w^{-}(x) w^{+}(y)}{|x-y|^{N+2\alpha}} \, dxdy = \int_{\R^{N}} K(x) f(s_0 w^{-}) s_{0}w^{-} \, dx. \label{a14}
\end{align}
\indent Dividing by $s_{0}^{2}>0$ in (\ref{a14}), we have
\begin{align*}\begin{split}
\|w^{-}\|^{2} - \frac{t_{0}}{s_{0}} \iint_{\R^{2N}} \frac{w^{+}(x)w^{-}(y) + w^{-}(x) w^{+}(y)}{|x-y|^{N+2\alpha}} \, dxdy = \int_{\R^{N}} K(x) \frac{f(s_0 w^{-})}{s_{0}w^{-}}(w^{-})^{2} \, dx,
\end{split}\end{align*}
and by using the fact that $0<t_{0}\leq s_{0}$ we can see that
\begin{align}\begin{split}\label{a15}
\|w^{-}\|^{2} - \iint_{\R^{2N}} \frac{w^{+}(x)w^{-}(y) + w^{-}(x) w^{+}(y)}{|x-y|^{N+2\alpha}} \, dxdy \geq \int_{\R^{N}} K(x) \frac{f(s_0 w^{-})}{s_{0}w^{-}}(w^{-})^{2} \, dx.
\end{split}\end{align}
\indent Since $w\in \mathcal{M}$, we also have
\begin{align}\begin{split}\label{a16}
\|w^{-}\|^{2} - \iint_{\R^{2N}} \frac{w^{+}(x)w^{-}(y) + w^{-}(x) w^{+}(y)}{|x-y|^{N+2\alpha}} \, dxdy = \int_{\R^{N}} K(x) \frac{f(w^{-})}{w^{-}}(w^{-})^{2} \, dx.
\end{split}\end{align}
Putting together (\ref{a15}) and (\ref{a16}) we get
\begin{align*}
0\geq  \int_{\R^{N}} K(x) \left[\frac{f(s_0 w^{-})}{s_{0}w^{-}}(w^{-})^{2} -\frac{f(w^{-})}{w^{-}}(w^{-})^{2}\right] \, dx.
\end{align*}
\indent The above relation and assumption $(f_4)$ ensures that $0<t_{0}\leq s_{0}\leq 1$. \\
Now we prove that $t_{0}\geq 1$. Dividing by $t_{0}^{2}>0$ in (\ref{a13}), we have
\begin{align*}\begin{split}
\|w^{+}\|^{2} - \frac{s_{0}}{t_{0}} \iint_{\R^{2N}} \frac{w^{+}(x)w^{-}(y) + w^{-}(x) w^{+}(y)}{|x-y|^{N+2\alpha}} \, dxdy = \int_{\R^{N}} K(x) \frac{f(t_0 w^{+})}{t_{0}w^{+}}(w^{+})^{2} \, dx
\end{split}\end{align*}
and by using $0<t_{0}\leq s_{0}$ we deduce that
\begin{align}\begin{split}\label{a17}
\|w^{+}\|^{2} - \iint_{\R^{2N}} \frac{w^{+}(x)w^{-}(y) + w^{-}(x) w^{+}(y)}{|x-y|^{N+2\alpha}} \, dxdy \leq \int_{\R^{N}} K(x) \frac{f(t_0 w^{+})}{t_{0}w^{+}}(w^{+})^{2} \, dx.
\end{split}\end{align}
Since $w\in \mathcal{M}$, we also have
\begin{align}\begin{split}\label{a18}
\|w^{+}\|^{2} - \iint_{\R^{2N}} \frac{w^{+}(x)w^{-}(y) + w^{-}(x) w^{+}(y)}{|x-y|^{N+2\alpha}} \, dxdy = \int_{\R^{N}} K(x) \frac{f(w^{+})}{w^{+}}(w^{+})^{2} \, dx.
\end{split}\end{align}
Putting together (\ref{a17}) and (\ref{a18}) we get
\begin{align*}
0\geq  \int_{\R^{N}} K(x) \left[\frac{f(w^{+})}{w^{+}}(w^{+})^{2} -\frac{f(t_{0}w^{+})}{t_{0}w^{+}}(w^{+})^{2}\right] \, dx.
\end{align*}
\indent By $(f_4)$ it follows that $t_{0}\geq1$. Consequently $t_{0}=s_{0}=1$, and this proves that $(1,1)$ is the unique critical point of $h^{w}$ with positive coordinates.  \\
\indent Let $u^{\pm}\in \X$ be such that $u^{\pm} \neq 0$, and let $(t_1, s_1), (t_2, s_2)$ be critical points of $h^{u}$ with positive coordinates. By $(i)$ it follows that
\begin{equation*}
w_1= t_1 u^{+} + s_1 u^{-} \in \mathcal{M} \, \mbox{ and } \, w_2= t_2 u^{+} + s_2 u^{-} \in \mathcal{M}.
\end{equation*}
\indent We notice that $w_2$ can be written as
\begin{equation*}
w_2 = \left(\frac{t_2}{t_1}\right) t_1 u^{+} + \left(\frac{s_2}{s_1}\right) s_1 u^{-} = \frac{t_2}{t_1} w_1^{+} + \frac{s_2}{s_1} w_1^{-} \in \mathcal{M}.
\end{equation*}
Since $w_1\in \X$ is such that $w_1^{\pm}\neq 0$, we have that $\left({t_2}/{t_1}, {s_2}/{s_1}\right)$ is a critical point for $h^{w_1}$ with positive coordinates.\par
 On the other hand, since $w_1 \in \mathcal{M}$, we can conclude that ${t_2}/{t_1}= {s_2}/{s_1}=1$, which gives $t_1= t_2$ and $s_1= s_2$.\par
 Finally, we prove that $h^{u}$ has a maximum global point $(\bar{t}, \bar{s})\in (0, +\infty)\times (0, +\infty)$. Let $A^{+}\subset \supp u^{+}$ and $A^{-}\subset \supp u^{-}$ positive with finite measure. By assumption $(f_3)$ and the fact that $F(t)\geq 0$ for every $t\in \R$, it follows that
\begin{align*}
h^{u}(t,s) &\leq \frac{1}{2} \|tu^{+}+su^{-}\|^{2}-\int_{A^{+}} K(x) F(tu^{+})\, dx - \int_{A^{-}} K(x) F(su^{-})\, dx\\
&\leq  \frac{t^{2}}{2}\|u^{+}\|^{2}+\frac{s^{2}}{2}\|u^{-}\|^{2}-st \iint_{\R^{2N}} \frac{u^{+}(x)u^{-}(y)+u^{+}(y)u^{-}(x)}{|x-y|^{N+2\alpha}} \, dx dy \\
&-\int_{A^{+}} K(x) F(tu^{+})\, dx - \int_{A^{-}} K(x) F(su^{-})\, dx.
\end{align*}
\indent Let us suppose that $|t|\geq |s|>0$.
Then, by using the fact that $F(t)\geq 0$ for every $t\in \R$, we can see that
\begin{align*}
h^{u}(t, s)&\leq (t^{2}+s^{2}) \Bigl[ \frac{1}{2}\|u^{+}\|^{2}+\frac{1}{2}\|u^{-}\|^{2}-\frac{1}{2} \iint_{\R^{2N}} \frac{u^{+}(x)u^{-}(y)+u^{+}(y)u^{-}(x)}{|x-y|^{N+2\alpha}} \, dx dy\Bigr] \\
&-t^{2}\int_{A^{+}} K(x) \frac{F(tu^{+})}{(t u^{+})^{2}}(u^{+})^{2}\, dx.
\end{align*}
Condition $(f_3)$, Fatou's lemma and the fact that $0<t^{2}+s^{2}\leq 2 t^{2}$, ensure that
$$
\limsup_{|(t,s)|\rightarrow \infty}\frac{h^{u}(t,s)}{t^{2}+s^{2}}\leq C(u^{+}, u^{-})-\frac{1}{2}\liminf_{|t|\rightarrow \infty}\int_{A^{+}} K(x) \frac{F(tu^{+})}{(t u^{+})^{2}}(u^{+})^{2}\, dx=-\infty,
$$
where $C(u^{+}, u^{-})> 0$ is a constant depending only on $u^{+}$ and $u^{-}$.\\
\indent Therefore
\begin{equation}\label{formuletta0}
\lim_{|(t,s)|\rightarrow \infty} h^{u}(t,s)= -\infty.
\end{equation}
By \eqref{formuletta0}, and recalling that $h^{u}$ is a continuous function, we deduce that $h^{u}$ has a maximum global point $(\bar{t}, \bar{s})\in (0, +\infty)\times (0, +\infty)$. \\
\indent The linearity of $F$ and the positivity of $K$ yield
\begin{equation}\label{formuletta}
\int_{\R^{N}} K(x) (F(tu^{+})+ F(su^{-}))\, dx = \int_{\R^{N}} K(x) F(tu^{+}+su^{-})\, dx.
\end{equation}
By \eqref{formuletta} for all $u\in \X$ such that $u^{\pm}\neq 0$ and for every $t,s\geq 0$, it follows that
\begin{equation*}
J(tu^{+})+ J(su^{-})\leq J(tu^{+}+su^{-}).
\end{equation*}
\indent So, for every $u\in \X$ such that $u^{\pm}\neq 0$ one has
\begin{equation*}
h^{u}(t,0) + h^{u}(0,s) \leq h^{u}(t,s),
\end{equation*}
for every $t,s\geq 0$.\par
Then,
\begin{align*}
\max_{t\geq 0} h^{u}(t,0) < \max_{t,s>0} h^{u}(t,s) \, \mbox{ and } \, \max_{s\geq 0} h^{u}(0,s)<\max_{t,s>0} h^{u}(t,s),
\end{align*}
and this proves that $(\bar{t}, \bar{s})\in (0, +\infty)\times (0, +\infty)$.
\bigskip

\noindent
$(iii)$ By Lemma \ref{lemz1}-$(a)$ we easily have that
\begin{align*}
&\Phi_{1}^{u}(r, s_{-})= \frac{\partial h^{u}}{\partial t}(r, s_{-})>0 \, \mbox{ if } \, r\in (0, t_{+}) \\
&\Phi_{1}^{u}(t_{+}, s_{-})= \frac{\partial h^{u}}{\partial t}(t_{+}, s_{-})=0 \\
&\Phi_{1}^{u}(r, s_{-})= \frac{\partial h^{u}}{\partial t}(r, s_{-})>0 \, \mbox{ if } \, r\in(t_{+}, +\infty).
\end{align*}
Therefore (\ref{mapa+-}) holds true. The proof of Lemma \ref{lemz11} is now complete.
\end{proof}

\begin{lem}\label{lemz12}
If $\{u_{n}\}_{n\in \N}\subset \mathcal{M}$ and $u_{n}\rightharpoonup u$ in $\X$, then $u\in \X$ and $u^{\pm}\neq 0$.
\end{lem}

\begin{proof}
Let us observe that there is $\beta>0$ such that
\begin{equation}\label{b17}
\beta \leq \|v^{\pm}\| \quad \forall v\in \mathcal{M}.
\end{equation}
Indeed, if $v\in \mathcal{M}$, then
\begin{equation*}
\|v^{\pm}\|^{2} \leq \int_{\R^{N}} K(x) f(v^{\pm})v^{\pm} \, dx.
\end{equation*}
\indent Assume that $(h_3)$ holds true. Then by using $(f_{1})$, $(f_{2})$, and the Sobolev inequality, we can see that given $\varepsilon>0$ there exists a positive constant $C_{\varepsilon}$ such that
\begin{align}\begin{split}\label{a22}
\|v^{\pm}\|^{2} &\leq \int_{\R^{N}} K(x) f(v^{\pm})v^{\pm} \, dx \\
&\leq \varepsilon \left\|{K}/{V}\right\|_{L^{\infty}(\R^{N})} \int_{\R^{N}} V(x) (v^{\pm})^{2} \, dx + C_{\varepsilon} C_{*}\|K\|_{L^{\infty}(\R^{N})} \|v^{\pm}\|^{2^{*}_{\alpha}}\\
&\leq \varepsilon \left\|{K}/{V}\right\|_{L^{\infty}(\R^{N})} \|v^{\pm}\|^{2} + C_{\varepsilon}C_{*} \|K\|_{L^{\infty}(\R^{N})} \|v^{\pm}\|^{2^{*}_{\alpha}}.
\end{split}\end{align}
Choosing $$\varepsilon\in \left(0, \frac{1}{\left\|{K}/{V}\right\|_{L^{\infty}(\R^{N})}}\right),$$ there exists a positive constant $\beta_1$ such that $\|v^{\pm}\|>\beta_1$.\par
Analogously, by assuming that $(h_4)$ holds, conditions $(\tilde{f}_{1})$ and $(f_{2})$, the Sobolev embedding result and the H\"older inequality ensure that
\begin{align}\begin{split}\label{a22}
\|v^{\pm}\|^{2} &\leq \int_{\R^{N}} K(x) f(v^{\pm})v^{\pm} \, dx \\
&\leq C_1 \varepsilon \|v^{\pm}\|^{2} + C_{1}C_{*} (\varepsilon+ C_2 \|K\|_{L^{\infty}(\R^{N})}) \|v^{\pm}\|^{2^{*}_{\alpha}} + \|K\|_{L^{\frac{2^{*}_{\alpha}}{2^{*}_{\alpha}-m}}(\B_{R}(0))} C_{*} \|v^{\pm}\|^{m}.
\end{split}\end{align}
Since $m\in (2, 2^{*}_{\alpha})$, we can choose $\varepsilon$ sufficiently small such that it is possible to find a positive constant $\beta_2$ such that $\|v^{\pm}\|>\beta_2$. \par
Hence, if we set $\beta:= \min \{\beta_1, \beta_2 \}$ inequality (\ref{b17}) immediately holds.\par
\smallskip
\noindent So, if $\{u_{n}\}_{n\in\N}\subset \mathcal{M}$, we have
\begin{equation}\label{b19}
\beta^{2} \leq \int_{\R^{N}} K(x) f(u_{n}^{\pm}) u_{n}^{\pm} \, dx \,,\quad \forall n\in \N.
\end{equation}
 Since $u_{n}\rightharpoonup u$ in $\X$, bearing in mind Lemma \ref{prop2.2}, we can pass to the limit in (\ref{b19}) as $n\rightarrow \infty$.\par
  More precisely, by using Lemma \ref{lem2.1}, it follows that
\begin{equation*}
0<\beta^{2} \leq \int_{\R^{N}} K(x) f(u^{\pm}) u^{\pm} \, dx.
\end{equation*}
Thus $u\in \X$ and $u^{\pm}\neq 0$. The proof is now complete.
\end{proof}

Let us denote by $c_{\infty}$ the number
\begin{equation*}
c_{\infty}:= \inf_{u\in \mathcal{M}} J(u).
\end{equation*}
Since $\mathcal{M}\subset \mathcal{N}$, we deduce
\begin{equation}\label{z10}
c_{\infty}\geq d_{\infty}>0.
\end{equation}


\section{Proof of Theorem \ref{thmf}}\label{fine}

\noindent In this section we prove the existence of energy nodal weak solutions by using minimization arguments and a variant of the Deformation lemma. We start by proving the existence of a minimum point of the functional $J$ in $\mathcal{M}$. \par
Let $\{u_{n}\}_{n\in \N}\subset \mathcal{M}$ be such that
\begin{equation}\label{c1}
J(u_{n})\rightarrow c_{\infty} \quad \mbox{ in } \R.
\end{equation}
Our aim is to prove that $\{u_{n}\}_{n\in\N}$ is bounded in $\X$.\par
 Indeed, assume by contradiction that that there exists a subsequence, denoted again by $\{u_{n}\}_{n\in\N}$, such that $\|u_{n}\|\rightarrow +\infty$ as $n\rightarrow \infty$. Thus, let us define $$v_{n}:= \frac{u_{n}}{\|u_{n}\|},$$ for every $n\in \N$. Since $\{v_{n}\}_{n\in\N}$ is bounded in $\X$, due to the reflexivity of $\X$, there exists $v\in \X$ such that
\begin{equation}\label{c3}
v_{n}\rightharpoonup v \quad \mbox{ in } \X.
\end{equation}
\indent Moreover, in virtue of Lemma \ref{cont}, it follows that
\begin{equation}\label{c4}
v_{n}(x) \rightarrow v(x) \, \mbox{ a.e. in } \R^{N}.
\end{equation}
\noindent By Lemma \ref{lemz11}-$(i)$ and $\{u_{n}\}_{n\in\N}\subset \mathcal{M}$, we have that $t_{+}(v_{n})= s_{-}(v_{n})=\|u_{n}\|$ and
\begin{align}\label{c5}
J(u_{n})&= J(\|u_{n}\| v_{n})\geq J(tv_{n})\nonumber\\
&= \frac{t^{2}}{2} \|v_n\|^{2} - \int_{\R^{N}}K(x)F(tv_{n})\, dx \\
&= \frac{t^{2}}{2} - \int_{\R^{N}}K(x)F(tv_{n})\, dx,\nonumber
\end{align}
for every $t>0$ and $n\in \N$. \par
Suppose that $v=0$. Taking into account (\ref{c3}) and Lemma \ref{lem2.1} we get
\begin{equation}\label{c6}
\int_{\R^{N}} K(x) F(tv_{n}) \rightarrow 0, \quad \forall t>0.
\end{equation}
By passing to the limit in \eqref{c5}, as $n\rightarrow \infty$, and combining (\ref{c1}) and (\ref{c6}) we have
\begin{equation*}
c_{\infty}\geq \frac{t^{2}}{2}, \quad \forall t>0
\end{equation*}
which is a contradiction.\par
Hence, $v\neq 0$. Taking into account the definitions of $J$ and $\{v_{n}\}_{n\in\N}$ we have
\begin{equation}\label{c7}
\frac{J(u_{n})}{\|u_{n}\|^{2}} = \frac{1}{2} - \int_{\R^{N}} K(x) \frac{F(v_{n}\|u_{n}\|)}{(v_{n}\|u_{n}\|)^{2}} (v_{n})^{2}\, dx.
\end{equation}
Now, since $v\neq 0$ and $\|u_{n}\|\rightarrow +\infty$, by using (\ref{c4}) in addition to $(f_3)$, the Fatou's lemma ensures that
\begin{equation}\label{c8}
\int_{\R^{N}} K(x)\frac{F(v_{n} \|u_{n}\|)}{(v_{n} \|u_{n}\|)^{2}}v_{n}^{2} \, dx \rightarrow +\infty.
\end{equation}
 \indent Since (\ref{c8}) holds true, bearing in mind (\ref{c1}) and passing to the limit in (\ref{c7}), as $n\rightarrow \infty$, we have a contradiction.\par
Therefore $\{u_{n}\}_{n\in\N}\subset \X$ is a bounded subsequence. As a consequence, there exists $u\in \X$ such that
\begin{equation}\label{c9}
u_{n}\rightharpoonup u \quad \mbox{ in } \X.
\end{equation}
By Lemma \ref{lemz12} it follows that $u^{\pm}\neq 0$. Moreover, by Lemma \ref{lemz11}, there are two constants $t_{+}, s_{-}>0$ such that
\begin{equation}\label{c10}
t_{+}u^{+}+s_{-}u^{-}\in \mathcal{M}.
\end{equation}
\indent Now, our aim is to prove that $t_{+}, s_{-}\in (0, 1]$. By (\ref{c9}) and Lemma \ref{lem2.1} we have
\begin{align}\label{c12}
\int_{\R^{N}} K(x) f(u_{n}^{\pm}) u_{n}^{\pm} \, dx \rightarrow \int_{\R^{N}} K(x) f(u^{\pm}) u^{\pm}\, dx
\end{align}
and
\begin{align}\label{c13}
\int_{\R^{N}} K(x) F(u_{n}^{\pm}) \, dx \rightarrow \int_{\R^{N}} K(x) F(u^{\pm}) \, dx.
\end{align}
\indent Recalling that $\{u_n\}_{n\in \N}\subset \mathcal{M}$, by using (\ref{c9}) and (\ref{c12}), Fatou's lemma ensures that
\begin{align}\label{c14}
&\langle J'(u), u^{\pm} \rangle \nonumber\\
&=\|u^{\pm}\|^{2}-\iint_{\R^{2N}} \frac{u^{+}(x)u^{-}(y)+u^{-}(x)u^{+}(y)}{|x-y|^{N+2\alpha}}\, dx dy-\int_{\R^{N}} K(x) f(u^{\pm})u^{\pm} dx \nonumber \\
&\leq \liminf_{n\rightarrow \infty} \langle J'(u_{n}), u_{n}^{\pm} \rangle =0.
\end{align}
Let us assume $0<t_{+}<s_{-}$. By (\ref{c10}) we deduce
\begin{equation*}
s^{2}_{-}\|u^{-}\|^{2}-t_{+}s_{-}\iint_{\R^{2N}} \frac{u^{+}(x)u^{-}(y)+u^{-}(x)u^{+}(y)}{|x-y|^{N+2\alpha}} \, dx dy= \int_{\R^{N}} K(x) f(s_{-}u^{-})\, s_{-}u^{-} \, dx,
\end{equation*}
and by using $t_{+}<s_{-}$ we obtain
\begin{equation}\label{c15}
\|u^{-}\|^{2}-\iint_{\R^{2N}} \frac{u^{-}(x)u^{+}(y)+u^{-}(y)u^{+}(x)}{|x-y|^{N+2\alpha}} \, dx dy\geq \int_{\supp u^{-}} K(x) \frac{f(s_{-}u^{-})}{s_{-}u^{-}} (u^{-})^{2}\, dx.
\end{equation}
\indent By (\ref{c14}) we have
\begin{equation}\label{c16}
\|u^{-}\|^{2}-\iint_{\R^{2N}} \frac{u^{-}(x)u^{+}(y)+u^{-}(y)u^{+}(x)}{|x-y|^{N+2\alpha}} \, dx dy \leq \int_{\supp u^{-}} K(x) \frac{f(u^{-})}{u^{-}}(u^{-})^{2}\, dx.
\end{equation}
Putting together (\ref{c15}) and (\ref{c16}) we can deduce that
\begin{equation}\label{c17}
0\geq \int_{\supp u^{-}} K(x) \left[ \frac{f(s_{-}u^{-})}{s_{-}u^{-}} - \frac{f(u^{-})}{u^{-}}\right] (u^{-})^{2} \, dx,
\end{equation}
which yields $s_{-}\in (0, 1]$ in virtue of $(f_4)$. Similarly, we can show that $t_{+}\in (0, 1]$. \par
\indent Now, we prove that
\begin{equation}\label{c18}
J(t_{+}u^{+}+s_{-}u^{-})= c_{\infty}.
\end{equation}
\indent By using the definitions of $c_{\infty}$, $t_{+}, s_{-}\in (0, 1]$, exploiting condition $(f_4)$, and taking into account relations (\ref{c10}), (\ref{c12}) and (\ref{c13}), we get
\begin{align}\begin{split}\label{c19}
c_{\infty}&\leq J(t_{+}u^{+}+s_{-}u^{-})\\
&= J(t_{+}u^{+}+s_{-}u^{-})-\frac{1}{2}\langle J'(t_{+}u^{+}+s_{-}u^{-}), t_{+}u^{+}+s_{-}u^{-}\rangle \\
&=\int_{\R^{N}} K(x) \left[ \frac{1}{2} f(t_{+}u^{+}+s_{-}u^{-}) (t_{+}u^{+}+s_{-}u^{-}) - F(t_{+}u^{+}+s_{-}u^{-}) \right]\, dx \\
&=\int_{\R^{N}} K(x) \left[ \frac{1}{2} f(t_{+}u^{+})(t_{+}u^{+}) - F(t_{+}u^{+}) \right] \, dx+\int_{\R^{N}} K(x) \left[ \frac{1}{2} f(s_{-}u^{-})(s_{-}u^{-}) - F(s_{-}u^{-}) \right]\, dx \\
&\leq \int_{\R^{N}} K(x) \left[ \frac{1}{2} f(u^{+})(u^{+}) - F(u^{+}) \right] \, dx+\int_{\R^{N}} K(x) \left[ \frac{1}{2} f(u^{-})(u^{-}) - F(u^{-}) \right]\, dx \nonumber \\
&=\int_{\R^{N}} K(x) \left[\frac{1}{2}f(u)u-F(u) \right]\, dx \nonumber\\
&= \lim_{n\rightarrow \infty} \int_{\R^{N}} K(x) \left[\frac{1}{2}f(u_{n})u_{n}-F(u_{n})\right]\, dx\\
&=\lim_{n\rightarrow \infty} \left[J(u_{n}) -\frac{1}{2}\langle J'(u_{n}), u_{n}\rangle \right]=c_{\infty}.
\end{split}\end{align}
Hence (\ref{c18}) holds true. Furthermore, the above calculation implies that $t_{+}=s_{-}=1$. \par
\indent Now, we prove that $u=u^{+}+u^{-}$ is a critical point of the functional $J$ arguing by contradiction. Thus, let us suppose that $J'(u)\neq 0$. By continuity there exist $\delta, \mu>0$ such that
\begin{equation}\label{c22}
\mu \leq |J'(v)|, \, \mbox{ since } \|v- u\|\leq 3\delta.
\end{equation}
\indent Define $\it{D}:= [\frac{1}{2}, \frac{3}{2}]\times [\frac{1}{2}, \frac{3}{2}]$ and $g: \it{D}\rightarrow \X^{\pm}$ by
\begin{equation*}
g(t,s):=tu^{+}+su^{-},
\end{equation*}
where $\X^{\pm}:= \{u\in \X : u^{\pm}\neq 0\}$.\par
 By Lemma \ref{lemz11} we deduce
\begin{align*}
&J(g(1,1))=c_{\infty}\\
&J(g(t,s))<c_{\infty} \,\quad \mbox{in } \it{D}\setminus\{(1,1)\}.
\end{align*}
Thus, we have
\begin{equation}\label{c23}
\beta:=\max_{(t,s)\in \partial \it{D}} J(g(t,s)) <c_{\infty}.
\end{equation}
\indent Now, we apply \cite[Theorem 2.3]{Willem} with $$\tilde{\mathcal{S}}:=\{v\in \X : \|v-u\|\leq \delta \},$$ and $c:=c_{\infty}$.\par
 Choosing $\varepsilon := \displaystyle\min \left\{ \frac{c_{\infty}-\beta}{4}, \frac{\mu \delta}{8}\right\}$, we deduce that there exists a deformation $\eta\in C([0,1]\times \X, \X)$ such that the following assertions hold:
\smallskip
\begin{compactenum}[(a)]
\item $\eta(t,v)= v$ if $v\in J^{-1}([c_{\infty}-2\varepsilon, c_{\infty}+2\varepsilon])$;
\item $J(\eta(1,v))\leq c_{\infty}-\varepsilon$ for each $v\in \X$ with $\|v-u\|\leq \delta$ and $J(v)\leq c_{\infty}+\varepsilon$;
\item $J(\eta(1,v))\leq J(v)$ for all $u\in \X$.
\end{compactenum}
\smallskip

\indent By $(b)$ and $(c)$ we conclude that
\begin{equation}\label{c24}
\max_{(t,s)\in \partial \it{D}} J(\eta(1, g(t,s))) <c_{\infty}.
\end{equation}
\indent To complete the proof it suffices to prove that
\begin{equation}\label{c25}
\eta(1, g(\it{D})) \cap \mathcal{M}\neq \emptyset.
\end{equation}
Indeed, the definition of $c_{\infty}$ and (\ref{c25}) contradict (\ref{c24}).\par
Hence, let us define the maps
\begin{align*}
&h(t,s):= \eta(1, g(t,s)), \\
&\psi_{0}(t,s):= \left(J'(g(t,1)) tu^{+}, J'(g(1,s)) su^{-} \right) \\
&\psi_{1}(t,s):= \left(\frac{1}{t} J'(h(t,1))h(t,1)^{+}, \frac{1}{s} J'(h(1,s))h(1,s)^{-}\right).
\end{align*}

\indent  By Lemma \ref{lemz11}-$(iii)$, the $C^{1}$-function $\gamma_{+}(t)=h^{u}(t,1)$ has a unique global maximum point $t=1$ (note that $t\gamma'_{+}(t)=\langle J'(g(t,1)), tu^{+}\rangle $). By density, given $\varepsilon>0$ small enough, there is $\gamma_{+, \varepsilon}\in C^{\infty}([\frac{1}{2}, \frac{3}{2}])$ such that $\|\gamma_{+} - \gamma_{+, \varepsilon}\|_{C^{1}([\frac{1}{2}, \frac{3}{2}])}<\varepsilon$ with $t_{+}$ being the unique maximum global point of $\gamma_{+, \varepsilon}$ in $[\frac{1}{2}, \frac{3}{2}]$. Therefore, $\|\gamma'_{+} - \gamma'_{+, \varepsilon}\|_{C([\frac{1}{2}, \frac{3}{2}])}<\varepsilon$, $\gamma'_{+, \varepsilon}(1)=0$ and $\gamma''_{+, \varepsilon}(1)<0$. Analogously, there exists $\gamma_{-, \varepsilon}\in C^{\infty}([\frac{1}{2}, \frac{3}{2}])$ such that $\|\gamma'_{-} - \gamma'_{-, \varepsilon}\|_{C([\frac{1}{2}, \frac{3}{2}])}<\varepsilon$, $\gamma'_{+, \varepsilon}(1)=0$ and $\gamma''_{+, \varepsilon}(1)<0$, where $\gamma_{-}(s)=h^{u}(1,s)$. \par

\indent Let us define $\psi_{\varepsilon}\in C^{\infty}(\it{D})$ by $\psi_{\varepsilon}(t,s):= (t \gamma'_{+, \varepsilon}(t), s \gamma'_{-, \varepsilon}(s))$ and note that $\|\psi_{\varepsilon} - \psi_{0} \|_{C(\it{D})}<\frac{3\sqrt{2}}{2} \varepsilon$, $(0,0) \not \in \psi_{\varepsilon}(\partial \it{D})$, and, $(0,0)$ is a regular value of $\psi_{\varepsilon}$ in $\it{D}$. On the other hand, $(1,1)$ is the unique solution of $\psi_{\varepsilon}(t,s)=(0,0)$ in $\it{D}$. By the definition of Brouwer's degree, we conclude that
\begin{equation*}
\rm{deg}(\psi_{0}, \it{D}, (0,0))= \rm{deg}(\psi_{\varepsilon}, \it{D}, (0,0))= \rm{sgn} \rm{Jac}(\psi_{\varepsilon})(1,1),
\end{equation*}
for $\varepsilon$ small enough.\par
Since
$$
\rm{Jac}(\psi_{\varepsilon})(1,1)= [\gamma'_{+, \varepsilon}(1)+\gamma''_{+, \varepsilon}(1)]\times [\gamma'_{-, \varepsilon}(1)+ \gamma''_{-, \varepsilon}(1)]= \gamma''_{+, \varepsilon}(1)\times \gamma''_{-, \varepsilon}(1)>0
$$
we obtain that
\begin{equation*}
\rm{deg}(\psi_{0}, \it{D}, (0,0))= \rm{sgn}[\gamma''_{+, \varepsilon}(1)\times \gamma''_{-, \varepsilon}(1)]=1,
\end{equation*}
where $\rm{Jac}(\psi_{\varepsilon})$ is the Jacobian determinant of $\psi_{\varepsilon}$ and $\rm{sign}$ denotes the sign function.\par
On the other hand, by (\ref{c23}) we have
\begin{align}\begin{split}\label{c26}
J(g(t,s))&\leq \beta\\
&<\frac{\beta + c_{\infty}}{2}\\
&=c_{\infty} - 2 \left( \frac{c_{\infty}-\beta}{4} \right)\\
&\leq c_{\infty}-2\varepsilon, \quad \forall (t,s)\in \partial \it{D}.
\end{split}\end{align}


By (\ref{c26}) and $(a)$ it follows that $g=h$ on $\partial \it{D}$. Therefore, $\psi_{1}=\psi_{0}$ on $\partial \it{D}$ and consequently
\begin{equation}\label{c27}
\rm{deg}(\psi_{1}, \it{D}, (0,0))= \rm{deg}(\psi_{0}, \it{D}, (0,0))=1,
\end{equation}
which shows that $\psi_{1}(t, s)=(0,0)$ for some $(t, s)\in \it{D}$.

Now, in order  to verify that (\ref{c25}) holds true, we prove that
\begin{equation}\label{fig}
\psi_{1}(1, 1)= \left(J'(h(t,1))h(1,1)^{+}, J'(h(1,1))h(1,1)^{-}\right)=0.
\end{equation}
As a matter of fact, (\ref{fig}) and the fact that $(1, 1)\in \it{D}$, yield $h(1, 1)=\eta(1, g(1, 1))\in \mathcal{M}$.

We argue as follows. If the zero $(t, s)$ of $\psi_{1}$ obtained above is equal to $(1, 1)$ there is nothing to do. On the other hand, if $(t, s)\neq (1, 1)$, we take $0<\delta_{1}<\min\{|t-1|, |s-1|\}$ and consider $$D_{1}:=\left[1-\frac{\delta_{1}}{2}, 1+\frac{\delta_{1}}{2}\right]\times \left[1-\frac{\delta_{1}}{2}, 1+\frac{\delta_{1}}{2}\right].$$

\indent Then, $(t, s)\in D \setminus D_{1}$. Hence, we can repeat for $D_{1}$ the same argument used for $D$, so that we can find a couple $(t_{1}, s_{1})\in D_{1}$ such that $\psi_{1}(t_{1}, s_{1}) = 0$. If $(t_{1}, s_{1})=(1, 1)$, there is nothing to prove. Otherwise, we can continue with this procedure and find in the $n$-th step that (\ref{fig}) holds, or produce a sequence $(t_{n}, s_{n})\in D_{n-1}\setminus D_{n}$ which converges to $(1, 1)$ and such that
\begin{equation}\label{italo}
\psi_{1}(t_{n}, s_{n})=0,\quad \mbox{ for every } n\in \N.
\end{equation}
Thus, taking the limit as $n \rightarrow \infty$ in (\ref{italo}) and using the continuity of $\psi_{1}$ we get (\ref{fig}).
Therefore, $u:=u^{+}+u^{-}$ is a critical point of $J$.\par

\smallskip

Finally, we consider the case when $f$ is odd. Clearly, the functional $\psi$ is even. From (\ref{z9}) and (\ref{z10}) we have that $\psi$ is bounded from below in $\mathbb{S}$. Taking into account Lemma \ref{prop2.2} and Lemma \ref{lem2.1}, we can infer that $\psi$ satisfies the Palais-Smale condition on $\mathbb{S}$. Then, by Proposition \ref{propz2} and \cite{Rab}, we can conclude that the functional $J$ has infinitely many critical points.

\bigskip

\indent {\bf Acknowledgements.} 
The authors warmly thank the anonymous referee for her/his useful and nice comments on the paper.
The manuscript was realized within the auspices of the INdAM - GNAMPA Projects 2017 titled: {\it Teoria e modelli per problemi non locali}.  


\addcontentsline{toc}{section}{\refname}

\begin{thebibliography}{777}

\bibitem{AM}
C.O. Alves, O.H. Miyagaki,
{\it Existence and concentration of solution for a class of fractional elliptic equation in $\R^{N}$ via penalization method},
Calc. Var. Partial Differential Equations {\bf 55} (2016), art. 47, 19 pp.

\bibitem{AS1}
C.O. Alves and M.A.S. Souto,
{\it Existence of solutions for a class of nonlinear Schr\"odinger equations with potential vanishing at infinity},
J. Differential Equations {\bf 254} (2013), no. 4, 1977--1991.

\bibitem{AS2}
C.O. Alves and M.A.S. Souto,
{\it Existence of least energy nodal solution for a Schr\"odinger-Poisson system in bounded domains},
Z. Angew. Math. Phys. {\bf 65} (2014), no. 6, 1153--1166.

\bibitem{AFM}
A. Ambrosetti, V. Felli and A. Malchiodi,
{\it Ground states of nonlinear Schr\"odinger equations with potentials vanishing at infinity},
J. Eur. Math. Soc. (JEMS) {\bf 7} (2005), no. 1, 117--144.

\bibitem{AW}
A. Ambrosetti and Z.-Q. Wang,
{\it Nonlinear Schr\"odinger equations with vanishing and decaying potentials},
Differential Integral Equations {\bf 18} (2005), no. 12, 1321--1332.

\bibitem{A0}
V. Ambrosio,
{\it Multiplicity of positive solutions for a class of fractional Schr\"odinger equations via penalization method}, 
Ann. Mat. Pura Appl. (4) {\bf 196} (2017), no. 6, 2043--2062.


\bibitem{A1}
V. Ambrosio, 
{\it Mountain pass solutions for the fractional Berestycki-Lions problem},  
Adv. Differential Equations {\bf 23} (2018), no. 5-6, 455--488.

\bibitem{A2}
V. Ambrosio,
{\it Multiplicity and concentration results for a fractional Choquard equation via penalization method}, 
Potential Analysis, DOI: 10.1007/s11118-017-9673-3 (in press). 


\bibitem{AmbFig}
V. Ambrosio and G. M. Figueiredo, 
{\it Ground state solutions for a fractional Schr\"odinger equation with critical growth}, 
Asymptotic Analysis {\bf 105} (2017), no. 3-4, pp. 159--191.

\bibitem{AH}
V. Ambrosio and H. Hajaiej, 
{\it Multiple solutions for a class of nonhomogeneous fractional Schr\"odinger equations in $\R^{N}$}, 
J. Dyn. Diff. Equat. (in press) DOI: 10.1007/s10884-017-9590-6.

\bibitem{AI2}
V. Ambrosio and T. Isernia,
{\it A multiplicity result for a fractional Kirchhoff equation in $\R^{N}$ with a general nonlinearity},  
Commun. Contemp. Math. https://doi.org/10.1142/S0219199717500547. 

\bibitem{AI}
V. Ambrosio and T. Isernia,
{\it Sign-changing solutions for a class of Schr\"odinger equations with vanishing potentials}, 
Rend. Lincei Mat. Appl. {\bf 29} (2018), 127--152;

\bibitem{AP}
G. Autuori and P. Pucci,
{\it Elliptic problems involving the fractional Laplacian in $\R^N$},
J. Differential Equations {\bf 255} (2013), 2340--2362.
	
\bibitem{BF}	
S. Barile and G.M. Figueiredo,
{\it Existence of least energy positive, negative and nodal solutions for a class of $p\&q$-problems with potentials vanishing at infinity},
J. Math. Anal. Appl. {\bf 427} (2015), no. 2, 1205--1233.	
	
\bibitem{BLW}
T. Bartsch, Z. Liu and T. Weth,
{\it Sign changing solutions of superlinear Schr\"odinger equations},
Comm. Partial Differential Equations {\bf 29} (2004), no. 1-2, 25--42.

\bibitem{BPW}
T. Bartsch, A. Pankov and Z.-Q. Wang,
{\it Nonlinear Schr\"odinger equations with steep potential well},
Comm. Contemp. Math. {\bf 4} (2001), 549--569.

\bibitem{BW}
T. Bartsch and T. Weth,
{\it Three nodal solutions of singularly perturbed elliptic equations on domains without topology},
Ann. Inst. H. Poincar\'e Anal. Non Lin\'eaire {\bf 22} (2005), no. 3, 259--281.

\bibitem{BWW}
T. Bartsch, T. Weth and M. Willem,
{\it Partial symmetry of least energy nodal solutions to some variational problems},
J. Anal. Math. {\bf 96} (2005), 1--18.

\bibitem{BGM2}
V. Benci, C.R. Grisanti and A.M. Micheletti,
{\it Existence of solutions for the nonlinear Schr\"odinger equations with $V (\infty) = 0$},
Contributions to non-linear analysis, 53--65, Progr. Nonlinear Differential Equations Appl., 66, Birkh\"auser, Basel, 2006.
	
\bibitem{BL1}
H. Berestycki and P.L. Lions,
{\it Nonlinear scalar field equations. I. Existence of a ground state},
Arch. Rational Mech. Anal. {\bf 82} (1983), no. 4, 313--345.

\bibitem{BVS}
D. Bonheure and J. Van Schaftingen,
{\it Ground states for the nonlinear Schr\"odinger equation with potential vanishing at infinity},
Ann. Mat. Pura Appl. (4) {\bf 189} (2010), no. 2, 273--301.

\bibitem{BV}
C. Bucur and E. Valdinoci, 
{\it Nonlocal diffusion and applications}, 
Lecture Notes of the Unione Matematica Italiana, 20. Springer, [Cham]; Unione Matematica Italiana, Bologna, 2016. xii+155 pp. ISBN: 978-3-319-28738-6; 978-3-319-28739-3.

\bibitem{Cab}
X. Cabr\'e and J. Sol\`a-Morales,
{\it Layer solutions in a half-space for boundary reactions},
Comm. Pure Appl. Math. {\bf 58} (2005),1678--1732.

\bibitem{CRS}
L.A. Caffarelli, J.-M. Roquejoffre and O. Savin,
{\it Non-local Minimal Surfaces},
Comm. Pure Appl. Math. {\bf 63} (2010), 1111--1144.

\bibitem{CSS}
L.A. Caffarelli, S. Salsa and L. Silvestre,
{\it Regularity estimates for the solution and the free boundary of the obstacle problem for the fractional Laplacian},
Invent. Math. {\bf 171} (2008), no. 2, 425--461.

\bibitem{CS}
L.A. Caffarelli and L. Silvestre,
{\it An extension problem related to the fractional Laplacian},
Comm. Partial Differential Equations {\bf 32} (2007), 1245--1260.

\bibitem{CCN}
A. Castro, J. Cossio and J. Neuberger,
{\it A sign-changing solution for a superlinear Dirichlet problem},
Rocky Mountain J. Math. {\bf 27} (1997), 1041--1053.

\bibitem{DDPDV}
J. D\'avila, M. del Pino, S. Dipierro and E. Valdinoci, 
{\it Concentration phenomena for the nonlocal Schr\"odinger equation with Dirichlet datum},
Anal. PDE {\bf 8} (2015), no. 5, 1165--1235.

\bibitem{DDPW}
J. D\'avila, M. del Pino and J. Wei, 
{\it Concentrating standing waves for the fractional nonlinear Schr\"odinger equation},
J. Differential Equations {\bf 256} (2014), no. 2, 858--892. 


\bibitem{DPV}
E. Di Nezza, G. Palatucci and E. Valdinoci,
{\it Hitchhiker's guide to the fractional Sobolev spaces},
Bull. Sci. math. {\bf 136} (2012), 521--573.

\bibitem{DiMV}
S. Dipierro, M. Medina and E. Valdinoci, 
{\it Fractional elliptic problems with critical growth in the whole of $\R^{n}$}, 
Appunti. Scuola Normale Superiore di Pisa (Nuova Serie) [Lecture Notes. Scuola Normale Superiore di Pisa (New Series)], 15. Edizioni della Normale, Pisa, 2017. viii+152 pp. ISBN: 978-88-7642-600-1; 978-88-7642-601-8.

\bibitem{DiPV}
S. Dipierro, G. Palatucci and E. Valdinoci,
{\it Existence and symmetry results for a Schr\"odinger type problem involving the fractional Laplacian},
Matematiche (Catania) {\bf 68} (2013), no. 1, 201--216.

	
\bibitem{FQT}
P. Felmer, A. Quaas and J. Tan,
{\it Positive solutions of the nonlinear {S}chr{\"o}dinger equation with the fractional {L}aplacian},
Proc. Roy. Soc. Edinburgh Sect. A {\bf 142} (2012), 1237--1262.

\bibitem{FS}
G.M. Figueiredo and G. Siciliano,
{\it A multiplicity result via Ljusternick-Schnirelmann category and Morse theory for a fractional Schr\"odinger equation in $\R^{N}$},
NoDEA Nonlinear Differential Equations Appl. 23 (2016), no. 2, Art. 12, 22 pp.

\bibitem{FJ}
G.M. Figueiredo and J.R. Santos Ju\'nior,
{\it Existence of a least energy nodal solution for a Schr\"odinger-Kirchhoff equation with potential vanishing at infinity},
Journal of Mathematical Physics. {\bf 56} (2015), 051506 18pp .

\bibitem{FPS}
A. Fiscella, P. Pucci and S. Saldi,
{\it Existence of entire solutions for Schr\"odinger-Hardy systems involving two fractional operators}, 
Nonlinear Anal. {\bf 158} (2017), 109--131. 

\bibitem{FLS}
R.L. Frank, E. Lenzmann and L. Silvestre,
{\it Uniqueness of radial solutions for the fractional Laplacian},
Commun. Pur. Appl. Math. {\bf 69} (2016), 1671--1726.


\bibitem{GR}
F. Gazzola and V. R\u adulescu,
{\it A nonsmooth critical point theory approach to some nonlinear elliptic equations in $\mathbb{R}^N$},
Differential Integral Equations {\bf 13} (2000), 47--60.

\bibitem{Isernia}
T. Isernia, 
{\it Positive solution for nonhomogeneous sublinear fractional equations in $\R^{N}$}, 
Complex Variables and Elliptic Equations,  DOI:10.1080/17476933.2017.1332052 (in press). 

\bibitem{ming1}
T. Kuusi, G. Mingione and Y. Sire,
{\it Nonlocal equations with measure data},
Communications in Mathematical Physics {\bf 337} (2015), 1317--1368.


\bibitem{Laskin1}
N. Laskin,
{\it Fractional quantum mechanics and L\'evy path integrals},
Phys. Lett. A {\bf 268} (2000), no. 4-6, 298--305. 81S40.

\bibitem{Laskin2}
N. Laskin,
{\it Fractional Schr\"odinger equation},
Phys. Rev. E (3) {\bf 66} (2002), no. 5, 056108, 7 pp.
81Q05.

\bibitem{Miranda}
C. Miranda,
{\it Un'osservazione sul teorema di Brouwer},
Boll. Unione Mat. Ital. {\bf 3} (2) (1940), 5--7.

\bibitem{MR}
G. Molica Bisci and V. R\u{a}dulescu,
{\it Ground state solutions of scalar field fractional Schr\"odinger equations},
Calc. Var. Partial Differential Equations {\bf 54} (2015), no. 3, 2985--3008.

\bibitem{MRS}
G. Molica Bisci, V. R\u{a}dulescu and R. Servadei,
{\it Variational Methods for Nonlocal Fractional Problems},
with a Foreword by Jean Mawhin, {\em Encyclopedia of Mathematics and its Applications}, {\em Cambridge University Press}, \textbf{162} Cambridge, 2016.


\bibitem{pu1}
P. Pucci and S. Saldi,
{\it Multiple solutions for an eigenvalue problem involving non-local elliptic $p$-Laplacian operators}, in Geometric Methods in PDE's - Springer INdAM Series - Vol. 11, G. Citti, M. Manfredini, D. Morbidelli, S. Polidoro, F. Uguzzoni Eds., pages 16.

\bibitem{pu2}
P. Pucci and S. Saldi,
{\it Critical stationary Kirchhoff equations in $\R^N$ involving nonlocal operators},
 Rev. Mat. Iberoam. {\bf 32} (2016), no. 1, 1--22.

\bibitem{seminalrabinowitz}
P.H. Rabinowitz,
{\it On a class of nonlinear Schr\"{o}dinger equations},
Z. Angew. Math. Phys. {\bf 43} (1992), 270--291.

\bibitem{Rab}
P. H. Rabinowitz,
{\it Minimax Methods in Critical Point Theory with Applications to Differential Equations},
CBMS Regional Conference Series in Mathematics Vol. 65 (American Mathematical Society, Providence, RI, 1986).

\bibitem{Secchi1}
S. Secchi,
{\it Ground state solutions for nonlinear fractional Schr\"odinger equations in $\R^{N}$},
J. Math. Phys. {\bf 54} (2013), 031501.



\bibitem{S}
W.A. Strauss,
{\it Existence of solitary waves in higher dimensions},
Comm. Math. Phys. {\bf 55} (1977), 149--162.


\bibitem{SW}
A. Szulkin and T. Weth,
{\it The method of Nehari manifold},
in Handbook of Nonconvex Analysis and Applications, edited by D. Y. Gao and D. Montreanu (International Press, Boston, 2010), pp. 597--632.

\bibitem{Willem}
M. Willem,
{\it Minimax Theorems}, Birkh\"{a}user, Basel, 1996.


\end{thebibliography}
\end{document}